\documentclass[11pt]{article}
\usepackage{amsmath}
\usepackage{amssymb}
\usepackage{graphicx}
\usepackage{color}
\usepackage{amsthm}
\usepackage{amsfonts}
\usepackage{amscd}
\usepackage{mathrsfs}
\usepackage{bm}
\usepackage{enumerate}
\usepackage{algorithmic, algorithm}

\renewcommand{\algorithmicrequire}{\textbf{Input:}} 
\renewcommand{\algorithmicensure}{\textbf{Output:}} 

\newcommand{\M}{\mathcal{M}}
\newcommand{\bx}{\mathbf{x}}
\newcommand{\bn}{\mathbf{n}}
\newcommand{\by}{\mathbf{y}}
\newcommand{\bw}{\mathbf{w}}
\newcommand{\bfp}{\mathbf{p}}

\newcommand{\bs}{\mathbf{s}}
\newcommand{\p}{\partial}

\newcommand{\e}{\epsilon}

\newcommand{\hr}{R\left(\frac{|\bx-\by|^2}{4t}\right)}
\newcommand{\hbr}{\bar{R}\left(\frac{|\bx-\by|^2}{4t}\right)}

\newcommand{\myref}[1]{(\ref{#1})}

\newcommand{\mathd}{\mathrm{d}}
\newcommand{\nd}[1]{\frac{\p #1}{\p \bn}}

\newcommand{\bV}{\mathbf{V}}
\newcommand{\bA}{\mathbf{A}}
\newcommand{\bu}{\mathbf{u}}

\newcommand{\hk}{R_t(\bx, \by)}
\newcommand{\rhk}{\bar{R}_t(\bx, \by)}
\newcommand{\hkpipj}{R_t({\bf p}_i, {\bf p}_j)}

\newcommand{\rhkpipj}{\bar{R}_t({\bf p}_i, {\bf p}_j)}
\newcommand{\rhkpisj}{\bar{R}_t({\bf p}_i, {\bf s}_j)}

\newcommand{\bfs}{{\bf s}}

\newcommand{\invt}{\frac{1}{t}}

\newtheorem{theorem}{\textbf{Theorem}}[section]
\newtheorem{lemma}{\textbf{Lemma}}[section]
\newtheorem{remark}{\textbf{Remark}}[section]

\newcommand{\R}{\mathbb{R}}
\numberwithin{equation}{section}

\makeatother

\begin{document}

\title{Point Integral Method for Solving Poisson-type Equations on Manifolds 
from Point Clouds with Convergence Guarantees}


\author{
Zhen Li%
\thanks{Mathematical Sciences Center, Tsinghua University, Beijing, China,
100084. \textit{Email: zli12@mails.tsinghua.edu.cn.}%
} \and Zuoqiang Shi%
\thanks{Mathematical Sciences Center, Tsinghua University, Beijing, China,
100084. \textit{Email: zqshi@math.tsinghua.edu.cn.}%
} \and Jian Sun %
\thanks{Mathematical Sciences Center, Tsinghua University, Beijing, China,
100084. \textit{Email: jsun@math.tsinghua.edu.cn.}%
}}

 \maketitle

\begin{abstract}
Partial differential equations (PDE) on manifolds arise in many areas, 
including mathematics and many applied fields. Among all kinds of PDEs, 
the Poisson-type equations including the standard Poisson equation and 
the related eigenproblem of the Laplace-Beltrami operator are 
of the most important. Due to the complicated geometrical structure of the manifold, it is 
difficult to get efficient numerical method to solve PDE on manifold. In the paper, 
we propose a method called point integral method (PIM) to solve the Poisson-type
equations from point clouds with convergence guarantees. In PIM, the key idea
is to derive the integral equations which approximates the Poisson-type equations 
and contains no derivatives but only the values of the unknown function. 
The latter makes the integral equation easy to be approximated from point cloud. 
In the paper, we explain the derivation of the integral equations, describe the point 
integral method and its implementation, and present the numerical experiments 
to demonstrate the convergence of PIM. 
\end{abstract}
\newpage
\section{Introduction}
Partial differential equations (PDE) on manifolds arise in many areas, including geometric flows
along manifolds in geometric analysis~\cite{caoyau}, movements of particles confined to surfaces 
in quantum mechanics~\cite{Schuster2003132, daCosta}, and distributions of physical or chemical quantities along 
interfaces in fluid mechanics~\cite{Defay}, among others.  It is well-known that one can extract
the geometric information of the manifolds by studying the behavior of 
partial differential equations or differential operators on the manifolds.  
This observation has been exploited both in mathematics, especially geometric analysis~\cite{yau}, and 
in applied fields, including machine learning~\cite{BelkinN05, Lafon04diffusion}, data analysis~\cite{Saito200868}, 
computer vision and image processing~\cite{Lindeberg}, geometric processing of 3D shapes~\cite{reuter06dna, Levy, OvsjanikovSG08}.  
Among all kinds of PDEs, the Poisson equation on manifolds and the related eigenproblem of the 
Laplace-Beltrami operator are of the most important, and have found applications in many fields. 
For instance, the eigensystem of the Laplace-Beltrami operator has been used for representing data in 
machine learning for dimensionality reduction~\cite{belkin2003led}, and for representing shapes in computer vision
and computer graphics for the analysis of images and 3D models~\cite{reuter06dna, OvsjanikovSG08}. 

Different approaches are available for solving the aforementioned partial differential equations, 
which we call the Poisson-type equations on manifolds, including Equation~\eqref{eq-neumann},
\eqref{eq-dirichlet}, \eqref{eigen-neumann}, and \eqref{eigen-dirichlet}, 
which will appear in Section~\ref{sec:problem}. 
If a manifold is represented by a mesh with nice elements, the finite element method (FEM) 
is effective for solving the Poisson-type equations on it.  
It is well-known that bad shaped elements may increase the condition number of the linear 
systems in FEM and hence reduce the accuracy of the solution~\cite{Shewchuk02whatis}. 
However, for a curved manifold, it is already very difficult to obtain a globally
consistent mesh~\cite{BoissonnatDG12}, let alone to generate a mesh with well-shaped 
elements~\cite{Dey:2006:CSR, Zhao:2001:FSR}.
To overcome the difficulty of triangulating manifolds, implicit representation (aka, level set
representation) can be used, where the differential equation is extended into the ambient 
space and discretized using a Euclidean grid~\cite{Bertalmio, Xu}. However, in the applications where
the ambient dimension is high, it is very expensive to lay down a Euclidean grid and solve
the equations on it. 

In this paper, we propose a method to solve the above Poisson-type equations on manifolds 
from point clouds with convergence guarantees. 
Unlike a mesh or a Euclidean grid, which may be difficult to generate or may introduce extra 
complexity, point cloud is the simplest way of representing a manifold, which is often made 
ready for use in practice and whose complexity depends only on the manifold itself. 
The main observation is that the Poisson equations can be 
approximated by certain integral equations which can be easily discretized and has a faithful approximation from 
point clouds. More precisely, we consider the Poisson equation with Neumann boundary condition:
\begin{eqnarray}
\label{eq-neumann-intro}
  \left\{\begin{array}{rl}
      -\Delta u(\bx)=f(\bx),&\bx\in \M,\\
      \frac{\p u}{\p \bn}(\bx)=g(\bx),& \bx\in \p \M,
\end{array}\right.
\end{eqnarray}
where $\M$ is a $k$ dimensional submanifold isometrically embedded in $\mathbb{R}^d$.
We show that its solution is well approximated by the solution of the following integral equation:
\begin{eqnarray}
\label{eq-integral-neumann-intro}
-\frac{1}{t}\int_\M (u(\bx)-u(\by))\hr \mathd \bx &=& \nonumber\\
\int_{\M} f(\bx) \hbr \mathd \bx 
&+& 2\int_{\p \M} g(\bx)\hbr \mathd \bx.\quad 
\end{eqnarray}
where the function $R(r):\mathbb{R}^+\rightarrow \mathbb{R}^+$ is either compactly supported or decays exponentially and
\begin{equation}
  \bar{R}(r)=\int_r^{+\infty}R(s)\mathd s.
\end{equation} 
One choice of the function $R$ is the well-known Gaussian. 
As the integral equation involves no derivatives of the unknown function $u$ but only the function values,  
it can be easily discretized from a point cloud which samples the underlying manifold.  
We call this method point integral method (PIM) as it only requires the approximation of integrals from 
the discrete representations. It has been shown that PIM has convergence guarantees for solving  
the Poisson-type equations on manifolds. The readers who are interested in the convergence analysis
are referred to our companion papers~\cite{SS14}. In this paper, we focus on 
describing the point integral method and its implementation, and presenting the numerical
experiments to demonstrate the convergence of PIM. 

\vspace{0.1in}
\noindent{\bf Related work:~}
Finite Element method is one of the most widely used method to solve the Poisson equations on surfaces. 
It has many good features. FEM converges fast: quadratically in $L^2$ and linearly in $H^1$ 
~\cite{Dziuk88}. FEM also works for solving
the eigensystem of Laplace-Beltrami operator~\cite{Strang73, Dodziuk76, Wardetzky06}. 
In computational aspect, for 
Poisson equation, the stiffness matrix obtained by FEM is symmetric, positive definite and sparse. 
There are lots of research on the fast solver for this kind of linear systems. Despite all these advantages, 
as we mentioned above, FEM requires a globally consistent mesh with well-shaped elements, which is very
difficult to generate for curved manifolds. 

Level set method embeds the manifolds into ambient spaces, and extends the differential equations into
ambient spaces, where the discretization of the differential equations can be done using Euclidean grids
of the ambient space~\cite{Bertalmio}. Level set method also has other advantages. For instance, with the 
help of implicit function, it becomes easy to estimate the normals and the curvatures of the manifold.
See the discussion in~\cite{Bertalmio, Bertalmio1} for more details. 
However, the main shortcoming of the level set method is that Euclidean grids are not intrinsic to 
the manifold and may introduce extra computational complexity, especially in the case where the ambient
dimension is high. 

There are other methods which solves PDEs on manifolds directly from point clouds. 
Liang and Zhao~\cite{Liang13} and Lai et al.~\cite{Lai13} propose the methods to
locally approximate the manifold and discretize the PDE using this local approximation,
and assemble them together into a global linear system for solving the PDE. Their methods 
are essentially FEM (level set method) but over a globally non-consistent mesh (implicit function). 
Although it may work very well in many empirical examples, it seems difficult to analyze the convergence 
behavior of their methods due to the non-consistency of the mesh or the implicit function. 

The point integral method is also related to the graph Laplacian with Gaussian weights. In~\cite{BelkinN05, 
Lafon04diffusion, Hein:2005:GMW, pcdlp2009}, it is shown that the graph Laplacian with Gaussian weights 
converges pointwisely to the Laplace-Beltrami operator when the vertices of the graph are assumed to 
sample the underlying manifold. The eigensystem of the weighted
graph Laplacian is shown to converge to the eigensystem of the Laplace-Beltrami operator when there 
is no boundary~\cite{CLEM_08, Dey:2010}, or there is Neumann boundary~\cite{Singer13}. 
Their proofs are done by relating the Laplacian to the heat operator, and thus it is essential to use 
the Gaussian kernel.

\vspace{0.1in}
\noindent{\bf Organization of the paper:~}
The remaining of the paper is organized as follows. In
Section~\ref{sec:problem}, we state the problems we want to solve. We derive the integral equations which approximate
the Poisson equations with Neumann and Dirichlet boundary conditions in Section~\ref{sec:neumann} and \ref{sec:dirichlet}
respectively. The details of discretizing the integral equations and its implementations are given in 
Section~\ref{sec:dis}. In Section~\ref{sec:weight}, we briefly describe the algorithm for estimating the volume weights
from point clouds. In Section~\ref{sec:examples}   we present several numerical results to show the performance of our method. 
At last, conclusion and remarks are made in Section~\ref{sec:discussion}.

\section{Statement of the problems}
\label{sec:problem}
In this paper, we consider the Poisson equation on a compact $k$-dimensional submanifold $\mathcal{M}$ 
in $\R^d$ with two kinds of boundary conditions: the Neumann boundary condition
\begin{equation}
\left\{\begin{array}{rl}
      -\Delta_\mathcal{M} u(\bx)=f(\bx),&\bx\in \mathcal{M} \\
      \frac{\p u}{\p \bn}(\bx)=g(\bx),& \bx\in \p \mathcal{M}
\end{array}
\right.
\tag{P1.a}
\label{eq-neumann}
\end{equation}
and the Dirichlet boundary condition, 
\begin{equation}
  \left\{\begin{array}{rl}
      -\Delta_\mathcal{M} u(\bx)=f(\bx),&\bx\in \mathcal{M} \\
      u(\bx)=g(\bx),&  \bx\in \p \mathcal{M}
\end{array}\right.
\tag{P2.a}
\label{eq-dirichlet}
\end{equation}
where $\Delta_\mathcal{M}$ is the Laplace-Beltrami operator on $\mathcal{M}$, 
and $\bn$ is the outward normal of $\p \mathcal{M}$. 
Let $g$ be the Riemannian metric tensor of $\mathcal{M}$. 
Given a local coordinate system $(x^1, x^2, \cdots, x^k)$,
the metric tensor $g$ can be represented by a matrix $(g_{ij})_{k\times k}$,
\begin{eqnarray}
  g_{ij}=\left<\frac{\p}{\p x^i},\frac{\p}{\p x^j}\right>,\quad i,j=1,\cdots,k.
\end{eqnarray}
Let $(g^{ij})_{k\times k}$ is the inverse matrix of $(g_{ij})_{k\times k}$, then it is well known that
the Laplace-Beltrami operator is
\begin{equation}
\Delta_{\mathcal{M}}  = \frac{1}{\sqrt{\det g}}\frac{\p}{\p x^i}(g^{ij}\sqrt{\det g} \frac{\p}{\p x^j})
\end{equation}
If $ \mathcal{M}$ is an open set in $\R^d$ with standard Euclidean metric, 
then $\Delta_\mathcal{M}$ becomes standard Laplace operator, i.e. $\Delta_{ \mathcal{M}} = \sum_{i=1}^d \frac{\p^2 }{\p {x^i}^2}$.

The other problem we consider is the following eigenproblem of the Laplace-Beltrami operator with 
the Neumann boundary 
\begin{equation}
\left\{\begin{array}{rcll}
      -\Delta_{\mathcal{M}} u(\bx) &=& \lambda u(x) ,&\bx\in  \mathcal{M} \\
      \frac{\p u}{\p \bn}(\bx) &=& 0,& \bx\in \p \mathcal{M}.
\end{array}
\right.  
\tag{P1.b}
\label{eigen-neumann}
\end{equation}
or the Dirichlet boundary
\begin{equation}
\left\{\begin{array}{rcll}
      -\Delta_{\mathcal{M}} u(\bx) &=& \lambda u(x) ,&\bx\in  \mathcal{M} \\
      u(\bx) &=& 0,& \bx\in \p \mathcal{M}.
\end{array}
\right.  
\tag{P2.b}
\label{eigen-dirichlet}
\end{equation}
A pair $(\lambda, u)$ solving the above equations is called an eigenvalue and the corresponding eigenfunction of the Laplace-Beltrami operator $\Delta_\mathcal{M}$. 
It is well known that the spectrum of the Laplace-Beltrami operator is discrete and all eigenvalues are nonnegative. 
Suppose $0=\lambda_0 \leq \lambda_1 \leq \lambda_2 \cdots $ are all eigenvalues listed in the ascending order and 
$\phi_0,~\phi_1,~\phi_2,~\cdots$ are their corresponding eigenfunctions. 
Then the problem we are interested in is how to compute these eigenvalues and the corresponding 
eigenfunctions from point clouds.  

So far, all the problems are stated in the continuous setting. Next, we will
introduce the discretization of the manifold $\mathcal{M}$.  Typically, the
explicit form of the submanifold $\mathcal{M}$ is not known. Instead, $\M$ is
represented by a set of sample points $P=\{\bfp_i|\; i=1,\cdots,n\}$,  and the
boundary of $\p \mathcal{M}$ is sampled by a subset
$S=\{\bs_i|\;i=1,\cdots,m\}\subset P$.
In addition, we may assume the following two vectors are given. 
The first one is $\bV = (V_1, \cdots, V_n)$ where $V_i$ is the volume weight of $\bfp_i$ on $\mathcal{M}$. 
The second one is $\bA= (A_1, \cdots, A_m)$ where $A_i$ is the volume weight of $\bs_i$ on $\p \mathcal{M}$. 
These two vectors are used to evaluate the integrals over $\mathcal{M}$ and $\p\mathcal{M}$.
For example, for any Lipschitz function $f$ on $\mathcal{M}$ and $g$ on $\p \mathcal{M}$, 
$\int_\mathcal{M} f(x) d\mu_x$ and  $\int_{\p \mathcal{M}} f(x) d\tau_x$ 
can be approximated by  $\sum_{i=1}^n f(\bfp_i) V_i$ and $\sum_{i=1}^m f(\bs_i) A_i$ respectively. 

\begin{remark}
We remark that if $\bV$ and $\bA$ are not given, they can be estimated as follows. 
\begin{itemize}
\item[(1)] If a mesh with the vertices $P$ approximating $\mathcal{M}$ is given, both weight vectors $\bV$ and $\bA$ can be easily
estimated from the given mesh by summing up the volume of the simplices incident to the vertices. 
One can obtain the input data which $h$-integral approximates $\M$ and $\p\M$ if the size of the elements in
the mesh is of order $h$ and the angle between the normal space of an element and the normal space of $\M$ at the
vertices of the element is of order $h^{1/2}$~\cite{Wardetzky06}. Note that there is no requirement on 
the shape of the elements in the mesh. 
\item[(2)] If the points in $P$$(S)$ are independent samples from uniform distribution on $\M$ ($\partial \M$),
then $\bV$$(\bA)$ can be taken as the constant vector $1/n$$(1/m)$. The integral of the functions
on $\M$$(\partial \M)$ can be estimated using Monte Carol method up to the volume of $\M$$(\partial \M)$;
\item[(3)] Finally, following~\cite{LuoSW09}, one can estimate the vectors $\bV$ and  $\bA$ by locally
approximating tangent spaces of $\M$ and $\p \M$, respectively. Specifically, for a point $p\in P$,
project the samples near to $p$ in $P$ onto the approximated tangent space at $p$ and take the volume
of the Voronoi cell of $p$ as its weight. In this way, one avoids constructing globally consistent meshes
for $\M$ and $\p\M$. 
\end{itemize}
\end{remark}

In the paper, we assume that the submanifold $\mathcal{M}$ and its boundary $\p \mathcal{M}$ are well resolved by the 
point set $P$ and $S$ in the sense that the integral of any $C^1$ function on $\mathcal{M}$ 
and $\p \mathcal{M}$ can be well approximated from the function values on $P$ and $S$ respectively. 
The issue becomes how to solve the Poisson equation on $(\mathcal{M}, \p \mathcal{M}$) from the sample points
$P$ and $S$ with guaranteed accuracy.

\section{The Neumann Boundary Condition}
\label{sec:neumann}
Let us consider the Poisson equation with the Neumann boundary condition given by~\eqref{eq-neumann}.
Given only unstructured point sets $P$ and $S$ without mesh information, it is difficult to discretize 
the Laplace-Beltrami operator, which is a differential operator. Our strategy is to first approximate 
the Poisson equation by an integral equation which involves no differentials but only the values 
of the unknown function, and then discretize the integral equation, which is relatively straightforward 
even without mesh.

We assume that the solution of the Neumann problem \myref{eq-neumann} is 
regular enough, at least belongs to $C^3(\mathcal{M})$. 
According to the theory of elliptic equations, this assumption could be true as long 
as $f$, $g$, the submanifold $\mathcal{M}$ and its boundary $\p\mathcal{M}$ are smooth enough.
Furthermore, we assume the function $R: \R^+ \rightarrow \R^+ $ is $C^1(\mathbb{R}^+)$ and  
$R(r) = 0$ for $\forall r >1$. Under these assumptions, we can have the following main theorem 
of this section. 

For a parameter $t$, let
\begin{equation*}
R_t(\bx,\by)=C_t\hr \quad \text{and} \quad \bar{R}_t(\bx,\by)=C_t\hbr,
\label{kernel-R}
\end{equation*}
where $C_t$ is a normalizing factor.
Recall that $\bar{R}(r) = \int_r^{+\infty} R(s)\mathd s$. 
Define the following operator for any function $u$ on $\M$ which makes the definition meaningful. 
\begin{equation}
L_{t}u(\bx) = \invt\int_\M \hk(u(\bx) - u(\by))d\mu_\by.
\label{eqn:laplace_integral}
\end{equation}
Let us call $L_t$ is the integral Laplace operator, which is clearly defined over $L^2(\M)$. 

In PIM, the approximate solution of the Neumann problem \eqref{eq-neumann} 
is obtained by solving the following integral equation with small $t$.
\begin{eqnarray}
\label{eq-integral-neumann}
L_tu(\by) = 2\int_{\p\mathcal{M}} g(\bx) \bar{R}_t(\bx,\by) \mathd \bx+\int_{\mathcal{M}} f(\bx) \bar{R}_t(\bx,\by) \mathd \bx.
\end{eqnarray}
Similarly, one can approximate the the eigenproblem of the Laplace-Beltrami operator with the Neumann boundary given by~\eqref{eigen-neumann}
by solving the following integral equation with small $t$.
\begin{eqnarray}
\label{eq-eigen-integral}
L_tu(\by) 
=\lambda\int_{\mathcal{M}} u(\bx) \bar{R}_t(\bx,\by) \mathd \mu_\bx.\quad 
\end{eqnarray}
Note that all the terms in \eqref{eq-integral-neumann} and \eqref{eq-eigen-integral}
are in the integral form, which is ready to be discretized by the point 
sets $P$ and $S$, and the associated volume weights $\bV$ and $\bA$. See Section~\ref{sec:dis} 
for the discretization of the above integral equations. 

Following theorem gives us an explanation that why the solution of the integral equation \eqref{eq-integral-neumann} could approximate the solution of 
the Neumann problem \eqref{eq-neumann}. 
\begin{theorem}
\label{theorem-local-error}
Let $u(\bx)$ be the solution of the Neumann problem given by \eqref{eq-neumann}, if $u\in C^3(\mathcal{M})$, then 
\begin{eqnarray}
\label{local-error}
\left\| -L_tu(\by)+2\int_{\p\mathcal{M}} g(\bx) \bar{R}_t(\bx,\by) \mathd \bx+\int_{\mathcal{M}} f(\bx) \bar{R}_t(\bx,\by) \mathd \bx\right\|_{L^2(\M)} \nonumber =O(t^{1/4}). \quad
\end{eqnarray}
\end{theorem}
\begin{remark}
Theorem \ref{theorem-local-error} by itself does not imply that the solution of the 
integral equation \eqref{eq-integral-neumann} respectively \eqref{eq-eigen-integral} converges to 
the solution of \eqref{eq-neumann} respectively \eqref{eigen-neumann} as $t\rightarrow 0$. The 
convergence requires the stability of the operator $L_t$. 
The rigorous proof for the convergence of the above integral equations is out of the scope of this paper. 
The interested readers are referred to the companion paper~\cite{SS14}.
\end{remark}

In the paper, we prove Theorem~\ref{theorem-local-error} for the case where $\mathcal{M}$ is an open set of
Euclidean space $\mathbb{R}^k$. For a general submanifold, the proof follows from the same idea 
but is technically more involved. The interested readers are referred to \cite{SS14}. 
In what follows, use $\Omega$ to denote the open set $\M$ in $\mathbb{R}^k$. First, we prove a technical lemma.
\begin{lemma}
  \label{prop-grad}
For any function $u\in C^3(\Omega)$, we have
\begin{eqnarray}
\label{eq-grad}
  \frac{1}{2t}\int_\Omega (\bx-\by)\cdot\nabla u(\bx)R_t(\bx,\by) \mathd \bx&=&\frac{1}{2t}\int_\Omega (u(\bx)-u(\by))
R_t(\bx,\by) \mathd \bx+
\frac{1}{2}\int_\Omega \Delta u\cdot\bar{R}_t(\bx,\by) \mathd \bx\nonumber\\
&&-\frac{1}{2}\int_{\p \Omega}((\bx-\by) \otimes\bn):\mathbf{H}_u(\bx)\bar{R}(\bx,\by) \mathd \bx+O(t^{1/2})
\end{eqnarray}
where $\mathbf{H}_u(\bx)$ is the Hessian matrix of $u$ at $\bx$, $\bn$ is the outer normal vector of $\p\Omega$.
\end{lemma}

\begin{proof}
The Taylor expansion of the function $u$ tells us that
\begin{eqnarray}
  u(\bx)-u(\by)=(\bx-\by)\cdot\nabla u(\bx)-\frac{1}{2}(\bx-\by)^T\mathbf{H}_u(\bx)(\bx-\by)+O(\|\bx-\by\|^3).
\end{eqnarray} 
Then, we have
\begin{eqnarray}
\label{eq-grad-taylor}
  &&\frac{1}{2t}\int_\Omega (\bx-\by)\cdot\nabla u(\bx)R_t(\bx,\by) \mathd \bx =\frac{1}{2t}\int_\Omega (u(\bx)-u(\by))R_t(\bx,\by) 
\mathd \bx\nonumber\\
&&+\frac{1}{4t}\int_\Omega (\bx-\by)^T\mathbf{H}_u(\bx)(\bx-\by)R_t(\bx,\by) \mathd \bx + O(t^{1/2}).
\end{eqnarray}
Here we use the fact that 
$\int_\Omega \|\bx-\by\|^nR_t(\bx,\by)\mathd \bx=O(t^{n/2})$.
Now, we turn to calculate the second term of \eqref{eq-grad-taylor}.
\begin{eqnarray}
\label{eq-hessian-euler}
&&\frac{1}{4t}\int_\Omega (\bx-\by)^T\mathbf{H}_u(\bx)(\bx-\by)R_t(\bx,\by) \mathd \bx\nonumber\\
&=&\frac{1}{4t}\int_\Omega (\bx_i-\by_i)(\bx_j-\by_j)\p_{ij} u(\bx)R_t(\bx,\by) \mathd \bx\nonumber\\
&=&-\frac{1}{2}\int_\Omega (\bx_i-\by_i)\p_{ij} u(\bx)\p_j\left(\bar{R}_t(\bx,\by)\right) \mathd \bx\nonumber\\
&=&\frac{1}{2}\int_\Omega \p_j(\bx_i-\by_i) \p_{ij} u(\bx)\bar{R}_t(\bx,\by) \mathd \bx
+\frac{1}{2}\int_\Omega (\bx_i-\by_i) \p_{ijj} u(\bx)\bar{R}_t(\bx,\by) \mathd \bx \nonumber\\
&&-\frac{1}{2}\int_{\p \Omega}(\bx_i-\by_i) \bn_j \p_{ij} u(\bx)\bar{R}_t(\bx,\by) \mathd \bx\nonumber\\
&=&\frac{1}{2}\int_\Omega \p_{ii} u(\bx)\bar{R}_t(\bx,\by) \mathd \bx
-\frac{1}{2}\int_{\p \Omega}(\bx_i-\by_i) \bn_j \p_{ij} u(\bx)\bar{R}_t(\bx,\by) \mathd \bx+O(t^{1/2})\nonumber\\
&=&\frac{1}{2}\int_\Omega \Delta u(\bx)\bar{R}_t(\bx,\by) \mathd \bx
-\frac{1}{2}\int_{\p \Omega}((\bx-\by) \otimes\bn):\mathbf{H}_u(\bx)\bar{R}_t(\bx,\by) \mathd \bx+O(t^{1/2}).
\end{eqnarray}
Here we used Einstein's summation convention. In the derivation of the second equality, we use the fact that 
\begin{eqnarray*}
  \p_j\bar{R}_t(\bx,\by)=-\frac{1}{2t}(\bx_j-\by_j)R_t(\bx,\by), 
\end{eqnarray*}
and for the fourth equality, we use the assumption that $u\in C^3(\Omega)$ to bound $\p_{ijj} u(\bx)$ and thus the second term is of the order $O(t^{1/2})$.
The lemma is proved by combining \eqref{eq-grad-taylor} and \eqref{eq-hessian-euler}.
\end{proof}

Now, we are ready to prove Theorem \ref{theorem-local-error}.
\begin{proof} {[\textit{Theorem \ref{theorem-local-error}}]}

Multiplying $\bar{R}_t(\bx,\by)$ on both sides of the Poisson equation, and by integral by parts, we have
\begin{eqnarray}
\label{eq-integral-part}
\int_\Omega \Delta u \cdot \bar{R}_t(\bx,\by) \mathd \bx
&=& -\int_\Omega \nabla u \cdot \nabla \bar{R}_t(\bx,\by) \mathd \bx
+\int_{\p \Omega} \frac{\p u}{\p \bn} \bar{R}_t(\bx,\by) \mathd \bx
\nonumber\\
&=& \frac{1}{2t}\int_\Omega (\bx-\by)\cdot\nabla u(\bx)R_t(\bx,\by) \mathd \bx
+\int_{\p \Omega} \frac{\p u}{\p \bn} \bar{R}_t(\bx,\by) \mathd \bx.
\end{eqnarray}
By Lemma \ref{prop-grad}, we have
\begin{eqnarray}
&&  \int_\Omega \Delta u \cdot \bar{R}_t(\bx,\by) \mathd \bx
= \frac{1}{2t}\int_\Omega (u(\bx)-u(\by))R_t(\bx,\by) \mathd \bx
+\frac{1}{2}\int_\Omega \Delta u\cdot\bar{R}_t(\bx,\by) \mathd \bx\nonumber\\
&&+ \int_{\p \Omega} \frac{\p u}{\p \bn} \bar{R}_t(\bx,\by) \mathd \bx
-\frac{1}{2}\int_{\p \Omega}((\bx-\by) \otimes\bn):\mathbf{H}_u(\bx)\bar{R}_t(\bx,\by) \mathd \bx
+O(t^{1/2}), \nonumber
\end{eqnarray}
which impies that
\begin{eqnarray}
  \int_\Omega \Delta u \cdot \bar{R}_t(\bx,\by) \mathd \bx
&=& \frac{1}{t}\int_\Omega (u(\bx)-u(\by))R_t(\bx,\by) \mathd \bx
+2\int_{\p \Omega} \frac{\p u}{\p \bn}(\bx) \bar{R}_t(\bx,\by) \mathd \bx
\nonumber\\
&&-\int_{\p \Omega}((\bx-\by) \otimes\bn):\mathbf{H}_u(\bx)\bar{R}_t(\bx,\by) \mathd \bx
+O(t^{1/2}).
\end{eqnarray}
Estimate the third term on the right hand side. 
\begin{eqnarray}
&&\int_{\Omega}\left|\int_{\p \Omega}((\bx-\by) \otimes\bn):\mathbf{H}_u(\bx)\bar{R}_t(\bx,\by) \mathd \bx\right|^2\mathd \by
\nonumber\\
&\le& \|\mathbf{H}_u(\bx)\|_{\infty} 
\int_{\Omega}\left(\int_{\p \Omega}\|\bx-\by\|\bar{R}_t(\bx,\by) \mathd \bx\right)^2\mathd \by\nonumber\\
&\le &\|\mathbf{H}_u(\bx)\|_{\infty}
\left\|\int_{\p \Omega}\|\bx-\by\|\bar{R}_t(\bx,\by) \mathd \bx\right\|_\infty 
\int_{\p \Omega}\left(\int_{\Omega}\|\bx-\by\|\bar{R}_t(\bx,\by) \mathd \by\right)\mathd \bx. \nonumber
\end{eqnarray}
Notice that 
\begin{eqnarray}
  \left\|\int_{\p \Omega}\|\bx-\by\|\bar{R}_t(\bx,\by) \mathd \bx\right\|_\infty=O(1)\quad \text{and} \quad 
\int_{\Omega}\|\bx-\by\|\bar{R}_t(\bx,\by) \mathd \by=O(t^{1/2}). 
\end{eqnarray}
Then we have
\begin{eqnarray}
\int_{\Omega}\left|\int_{\p \Omega}((\bx-\by) \otimes\bn):\mathbf{H}_u(\bx)\bar{R}_t(\bx,\by) \mathd \bx\right|^2\mathd \by=O(t^{1/2}).
\end{eqnarray}
Now if $u(\bx)$ be the solution of \eqref{eq-neumann}, it satisfies
\begin{eqnarray}
\label{eq-weak-neumann}
\int_\Omega \Delta u(\bx) \bar{R}_t(\bx,\by) \mathd \bx = 
-\int_\Omega f(\bx) \bar{R}_t(\bx,\by) \mathd \bx, \quad 
\forall \by\in \Omega.
\end{eqnarray}
We have proved the theorem.
\end{proof}


\begin{remark}
Theorem \ref{theorem-local-error} also holds for those $R$ which decays exponentially, such as the Gaussian function. 
The proof is similar. 
\end{remark}



\section{The Dirichlet boundary condition}
\label{sec:dirichlet}
In this section, we consider the Poisson equation with the Dirichlet boundary given by \eqref{eq-dirichlet}.  
We bridge the difference between the Neumann boundary and the Dirichlet boundary using the so-called the 
Robin boundary which mixes the Neumann boundary with the Dirichlet boundary. Specifically, consider the 
following problem
\begin{equation}
  \left\{\begin{array}{rl}
     -\Delta u(\bx) = f(\bx),& \bx\in \M,\\
      u(\bx)+\beta\frac{\p u}{\p \bn}(\bx) = g(\bx),&  \bx\in \p \M.
\end{array}\right.
\tag{P3.a}
\label{eq-robin}
\end{equation}
where $\beta>0$ is a parameter. 

As there is a Neumann component in the Robin boundary, we can solve the Robin problem~\eqref{eq-robin}
using the framework of solving Neumann problem~\eqref{eq-neumann} in Section~\ref{sec:neumann}.
Specifically, we approximate the Robin problem \eqref{eq-robin} by the following integral 
equation.
\begin{eqnarray}
\label{eqn:integral_robin}
L_tu(\by)
-\frac{2}{\beta}\int_{\p \M} (g(\bx)-u(\bx)) \bar{R}_t(\bx,\by) \mathd \tau_\bx
= \int_{\M} f(\bx) \bar{R}_t(\bx,\by) \mathd \mu_\bx.\quad\quad
\end{eqnarray}
Similarly, the corresponding eigenproblem of the Laplace-Beltrami operator with zero Robin boundary (i.e., $g=0$) can 
be approximated by the following integral equation.
\begin{eqnarray}
\label{eqn:eigen_integral_robin}
L_tu(\by)+\frac{2}{\beta}\int_{\p \M} u(\bx) \bar{R}_t(\bx,\by) \mathd \tau_\bx
= \lambda \int_{\M} u(\bx) \bar{R}_t(\bx,\by) \mathd \mu_\bx
\end{eqnarray}
Theoretically, it can be shown that the solution of \eqref{eq-robin} is a good 
approximation of the solution of the Dirichlet problem \eqref{eq-dirichlet} when $\beta$ is small. 
\begin{theorem}[\cite{SS14}] 
\label{thm:diff_D2R}
Suppose $u$ is the solution of the Dirichlet problem \eqref{eq-dirichlet} and 
$u_{R,\beta}$ is the solution of the Robin problem~\eqref{eq-robin}, then
  \begin{eqnarray}
    \|u-u_{R,\beta}\|_{H^1(\M)}\le C \beta^{1/2}\|u\|_{H^2(\M)}.
  \end{eqnarray}
\end{theorem}
Therefore, we can approximate the Dirichlet problem~\eqref{eq-dirichlet} and 
the corresponding eigenproblem using the integral equation~\eqref{eqn:integral_robin} and \eqref{eqn:eigen_integral_robin}
respectively by choosing small enough $\beta$. In a companion paper~\cite{SS14}, 
it is shown that the above approximations indeed converge as $t$ goes to $0$. Note that 
the choice of $\beta$ depends on $t$ and has to go to $0$ as $t$ goes to $0$.
Again the integral equations \eqref{eqn:integral_robin} and \eqref{eqn:eigen_integral_robin}
are ready to be discretized by the input data $(P, S, \bV, \bA)$. See Section~\ref{sec:dis} 
for the discretization of the above integral equations. 

\subsection{Iterative Solver based on Augmented Lagrangian Multiplier}
Notice that when $\beta$ is small, the linear system derived from the above approach becomes ill-conditioned.
We now propose an iterative method based on the Augmented Lagrange method (ALM) to alleviate the dependence
on the choice of $\beta$. We have not been able to show the convergence of this iterative method. However, 
we shall provide the theoretical evidence (see Section \ref{convergence-alm}) as well as the empirical 
evidence (see Section~\ref{sec:examples}) to support its convergence. 

\begin{algorithm*}[t]
\floatname{algorithm}{Procedure}
\caption{ALM for Dirichlet Problem}
\label{alg-diri-approx}
\begin{algorithmic}[1]
\STATE $k=0$, $w^0=0$.
\REPEAT  
\STATE Solving the following integral equation to get $v^k$,
\begin{eqnarray*}
L_tv^k(\by)&-&\frac{2}{\beta}\int_{\p \mathcal{M}} (g(\bx)-v^k(\bx)+\beta w^k(\bx)) \bar{R}_t(\bx,\by) \mathd \tau_\bx\
=\int_{\mathcal{M}} f(\bx) \bar{R}_t(\bx,\by) \mathd \mu_\bx.
\end{eqnarray*}
\STATE $w^{k+1}=w^k+\frac{1}{\beta}(g-(v^k|_{\p\M}))$, $k = k+1$
\UNTIL {$\|g-(v^{k-1}|_{\p \mathcal{M}})\| == 0$}
\STATE $u=v^k$
\end{algorithmic}
\end{algorithm*}

It is well known that the Dirichlet problem can be reformulated using the following
constrained variational problem:
\begin{eqnarray}
\label{opt-dirichelet}
  \min_{v\in H^1(\mathcal{M})} \frac{1}{2}\int_\mathcal{M} |\nabla v(\bx)|^2 \mathd \bx+\int_\mathcal{M} f(\bx) v(\bx) \mathd \bx,\quad \mbox{subject to:} \; \;\;
v(\bx)|_{\p\mathcal{M}}=g(\bx), 
\end{eqnarray}
and the ALM method can be used to solve the above problem as follows. 
Recall that for a constrained optimization problem
\begin{eqnarray}
  \min_{x} F(x),\quad \mbox{subject to:} \quad g(x)=0,
\end{eqnarray}
the ALM method solves it by the following iterative process
\begin{itemize}
\item $x^k=\arg\min_x L(x,w^k)$, where $  L(x,w)=F(x)+<w,g(x)>+\frac{1}{2\beta}\|g(x)\|^2$
\item $w^{k+1}=w^k+\frac{1}{\beta}g(x^k)$.
\end{itemize}
In essence, the ALM method solves a constrained problem by iteratively solving a 
sequence of unconstrained problem. It is well known that the convergence of ALM method
is robust to the choice of the parameter $\beta$. 

Applying the ALM method directly to the problem \eqref{opt-dirichelet}, 
the unconstrained problem which need to be solved iteratively is
\begin{eqnarray}
\label{alg-2}
\min_v &&\frac{1}{2}\int_\mathcal{M} |\nabla v(\bx)|^2 \mathd \mu_\bx+\int_\mathcal{M} f(\bx) \cdot v(\bx) \mathd \mu_\bx\nonumber\\
&&+\int_{\p\mathcal{M}}w^k(\bx)\cdot(g(\bx)-v(\bx))\mathd \tau_\bx
+\frac{1}{2\beta}\int_{\p\mathcal{M}}(g(\bx)-v(\bx))^2\mathd \tau_\bx.
\end{eqnarray}
Using the variational method, one can show that the solution to \eqref{alg-2}
is exactly the solution to the following Poisson equation with the Robin boundary:
\begin{eqnarray}
\label{eq-robin-ALM}
  \left\{\begin{array}{cl}
      \Delta v(\bx)=f(\bx),& \bx\in \mathcal{M},\\
      v(\bx)+\beta\frac{\p v}{\p \bn}(\bx)=g(\bx)+\beta w^k(\bx),&  \bx\in \p \mathcal{M}.
\end{array}\right.
\end{eqnarray}
Therefore, we have derived a method to solve the Dirichlet problem~\eqref{eq-dirichlet} by 
solving a sequence of the Robin problem in~\eqref{eq-robin-ALM} with 
the iteratively updated $w^k$. If the iterative process converges, we obtain
the correct boundary condition, i.e., $v(\bx) = g(\bx)$ for $\bx \in \partial \M$. 
In fact, $w^k$ converges to $\frac{\p v}{\p \bn}(\bx)$ for $\bx \in \partial \M$. 
So, it is not necessary to choose $\beta$ small to achieve the prescribed 
Dirichlet boundary. Finally, we summarize the above iterative method for 
solving the Dirichlet problem in Procedure~\ref{alg-diri-approx} (ALM for Dirichlet Problem).



\subsection{Towards a convergence analysis of the ALM iteration}
\label{convergence-alm}
We have not been able to proved the convergence of the above ALM iteration described in Procedure~\ref{alg-diri-approx}. 
However,  if the Robin problem \eqref{eq-robin-ALM} is exactly solved in each iteration, we can prove that the ALM iteration 
converges and hence the obtained solution converges the exact solution of the Dirichlet problem~\eqref{eq-dirichlet}.
\begin{theorem}
\label{theorem-converge-ALM}
  Let $u$ be the solution of the Dirichlet problem \myref{eq-dirichlet} 
and $v^k$ be the exact solution of \eqref{eq-robin-ALM} in the $k$th step. Then 
  \begin{eqnarray}
    \lim_{k\rightarrow \infty} \|u-v^k\|_{H^1(\M)}=0.
  \end{eqnarray}
\end{theorem}
\begin{proof}
Let $e^k(\bx)=v^k(\bx)-u(\bx)$, then $e^k$ satisfies the following equation
\begin{eqnarray}
\label{eq-error}
  \left\{\begin{array}{cl}
      \Delta e^k(\bx)=0,& \bx\in \M,\\
      e^k(\bx)+\beta\frac{\p e^k}{\p \bn}(\bx)=\beta\frac{\p e^{k-1}}{\p \bn}(\bx),& \bx\in \p \M.
\end{array}\right.
\end{eqnarray}
Multiplying $e^k$ to both sides of the above equation and integrate by part, 
\begin{eqnarray}
  -\int_\M |\nabla e^k|^2 \mathd \mu_\bx +\int_{\p \M} \nd{e^k}(\bx)e^k(\bx)\mathd \tau_\bx =0, 
\end{eqnarray}
which together with the boundary condition in~\eqref{eq-error} gives us that
\begin{eqnarray}
  \int_\M |\nabla e^k|^2 \mathd \mu_\bx +\beta\int_{\p \M} \left|\nd{e^k}\right|^2\mathd \tau_\bx =\beta\int_{\p \M} \nd{e^k}(\bx)\cdot\nd{e^{k-1}}(\bx)\mathd \tau_\bx.
\end{eqnarray}
Using the inequality 
\begin{eqnarray}
\int_{\p \M} \nd{e^k}(\bx)\cdot\nd{e^{k-1}}(\bx)\mathd \tau_\bx \le \frac{1}{2}\int_{\p \M} \left|\nd{e^k}\right|^2\mathd \tau_\bx
+\frac{1}{2}\int_{\p \M} \left|\nd{e^{k-1}}\right|^2\mathd \tau_\bx, 
\end{eqnarray}
we obtain
\begin{eqnarray}
\label{ineq-ek}
  \int_\M |\nabla e^k|^2 \mathd \mu_\bx +\frac{\beta}{2}\int_{\p \M} \left|\nd{e^k}\right|^2\mathd \tau_\bx \le \frac{\beta}{2}\int_{\p \M} \left|\nd{e^{k-1}}\right|^2\mathd \tau_\bx.
\end{eqnarray}
Now, let
\begin{eqnarray}
 a^k=\int_{\p \M} \left|\nd{e^k}\right|^2\mathd \tau_\bx,\quad k=1,2,\cdots 
\end{eqnarray}
Since $\int_\M |\nabla e^k|^2 \mathd \mu_\bx\ge 0$ and $\beta>0$, we have
\begin{eqnarray}
  0\le a^k\le a^{k-1},\quad \forall k\in \mathbb{N}
\end{eqnarray}
which means that $\{a^k\}$ is a monotonously decreasing sequence which is bounded below by 0. Then the sequence $\{a^k\}$ converges which implies that 
\begin{eqnarray}
  \lim_{k\rightarrow \infty}(a^{k-1}-a^k)=0.
\end{eqnarray}
Then, let $k$ go to $\infty$ in inequality \myref{ineq-ek}, we have
\begin{eqnarray}
 0\le  \lim_{k\rightarrow \infty} \int_\M |\nabla e^k|^2 \mathd \mu_\bx\le \frac{\beta}{2}\lim_{k\rightarrow \infty} (a^{k-1}-a^{k})=0
\end{eqnarray}
which implies that 
\begin{eqnarray}
 \lim_{k\rightarrow \infty} \int_\M |\nabla e^k|^2 \mathd \mu_\bx=0.
\end{eqnarray}
  Since $e^k$ satisfies \myref{eq-error}, we have 
  \begin{eqnarray}
    -\int_{\M} \nabla e^k\cdot \nabla \varphi\mathd \mu_\bx +\int_{\p\M}\nd{e^k}\varphi\mathd \tau_\bx=0,\quad \forall \varphi\in H^1(\M).
  \end{eqnarray}
Letting $k\rightarrow \infty$, we get 
\begin{eqnarray}
  \lim_{k\rightarrow \infty}\int_{\p\M}\nd{e^k}\varphi\mathd \tau_\bx= 0,\quad \forall \varphi\in H^1(\M).
\end{eqnarray}
By the boundary condition $(e^k+\beta \nd{e^k})|_{\p\M}=\beta\nd{e^{k-1}}|_{\p\M}$, we can derive 
\begin{eqnarray}
  \lim_{k\rightarrow \infty}\int_{\p\M}e^k\,\varphi\mathd \tau_\bx= 0,\quad \forall \varphi\in H^1(\M).
\end{eqnarray}
Then by the Poincare inequality, it is easy to show that 
\begin{eqnarray}
  \lim_{k\rightarrow \infty}\int_{\M}|e^k|^2\mathd \mu_\bx= 0
\end{eqnarray}
\end{proof}

\begin{remark}
Theorem \ref{theorem-converge-ALM} does not imply the convergence of Procedure \ref{alg-diri-approx}. 
In Procedure \ref{alg-diri-approx}, $u^k$ is not an exact solution of the Robin problem as 
is required in the above theorem. Theoretically, we can not exclude the case that the error in 
each iteration accumulates so that the iterative process diverges. 
However in the experiments we have done, Procedure~\ref{alg-diri-approx} always converges to the right
answer.   
\end{remark}

\section{Discretization of the integral equations}
\label{sec:dis}
In this section, we discretize the integral equations derived in Section \ref{sec:neumann} 
and Section \ref{sec:dirichlet} over the given input data $(P,S,\mathbf{V},\mathbf{A})$.
We assemble three matrices from the input data $(P,S,\mathbf{V},\mathbf{A})$ which are used to do numerical integral. 

The first matrix, denoted $\mathcal{L}$, is an $n\times n$ matrix defined as for any $\bfp_i, \bfp_j \in P$
\begin{eqnarray}
\label{dis_laplace}
\mathcal{L}_{ij} = 
\left\{\begin{array}{rl}
      -\invt \hkpipj V_j \text{~~if~~} i\neq j \\
		-\sum_{i\neq j} \mathcal{L}_{ij} \text{~~if~~} i = j. 
\end{array}
\right.  
\end{eqnarray}
For any function  $u\in C^1(M)$, let ${\bf u} = (u_1, \cdots, u_n)$ with $u_i = u(p_i)$ for any $p_i \in P$. Then 
$\mathcal{L}\bu$ is used to approximate the integral
\begin{eqnarray}
  \invt \int_\M \hk (u(\bx)-u(\by)) d\mu_\by
\end{eqnarray}
The matrix $\mathcal{L}$ was introduced as a discrete Laplace operator in~\cite{BelkinN05}. 

The second matrix, denoted $\mathcal{I}$, is also an $n\times n$ matrix defined as for any $\bfp_i, \bfp_j \in P$
\begin{eqnarray}
\label{domain_int}
\mathcal{I}_{ij} = \rhkpipj V_j 
\end{eqnarray}
For any function  $f\in C^1(M)$, let ${\bf f} = (f_1, \cdots, f_n)$ with $f_i = f(\bfp_i)$ for any $\bfp_i \in P$. Then
$\mathcal{I}{\bf f}$ is used to approximate the integral
\begin{eqnarray}
   \int_\M \rhk f(\by) d\mu_\by.
\end{eqnarray}

The third matrix, denoted $\mathcal{B}$, is an $n\times m$ matrix defined as for any $\bfp_i \in P$ and any $\bs_j \in S$
\begin{eqnarray}
\label{boundary_int}
\mathcal{B}_{ij} = \rhkpisj A_j 
\end{eqnarray}
For any function $g\in C^1(\p M)$, let ${\bf g} = (g_1, \cdots, g_n)$ with $g_i = g(\bs_i)$ for any $\bs_i \in S$
Then $\mathcal{B}{\bf g}$ is used to approximate the integral
\begin{eqnarray}
  \int_{\p \M} \rhk g(\by) d\tau_\by.
\end{eqnarray}

Now we are ready to describe the algorithms to solve the Poisson equation with different boundary conditions.
As we will see, they are simple and easy to implement. 
The following algorithm PoissonNeumann is used to solve the Poisson equation with the Neumann boundary. 
The derivation of the algorithm is described in the Section~\ref{sec:neumann}.  
\begin{algorithm}[h!]
\floatname{algorithm}{Algorithm}
\caption{ { PoissonNeumann($P, S, \bV, \bA, {\bf f}, {\bf g}, t$)} }
\begin{algorithmic}[1]
   \STATE Compute the matrices $\mathcal{L}, \mathcal{I}, \mathcal{B}$. 
	\STATE Set ${\bf b } = 2 \mathcal{B} {\bf g} + \mathcal{I} {\bf f}$. 
	\STATE Solve the linear system $\mathcal{L} {\bf u} = {\bf b}$ and obtain ${\bf u} = (u_1, \cdots, u_n)$.
	\STATE Output ${\bf u}$. 
\end{algorithmic}
\end{algorithm}

The eigenvalues and the eigenfunctions of the Laplace-Beltrami operator with the Neumann boundary condition are approximated 
by that of the generalized eigenproblem $\mathcal{L}{\bf v} = \gamma \mathcal{I}{\bf v} $.  Specifically, 
the $k$th smallest eigenvalue $\gamma_k$ and its corresponding ${\bf v}_k$ are used to approximate $\lambda_k$  and $\phi_k$ respectively. 
See Algorithm $3$ EigenNeumann. 
\begin{algorithm}[h!]
\floatname{algorithm}{Algorithm}
\caption{ { EigenNeumann($P, S, \bV, \bA, t$)} }
\begin{algorithmic}[1]
   \STATE Compute the matrices $\mathcal{L}, \mathcal{I}$. 
	\STATE Solve the generalized eigenproblem $\mathcal{L} {\bf v} =  \gamma \mathcal{I}{\bf v}$
 and obtain the eigenvalues
	$0=\gamma_0\leq \gamma_1 \leq \gamma_2, \cdots$ and the corresponding eigenvectors ${\bf v}_0, ~{\bf v}_1, ~{\bf v}_2, \cdots$
	\STATE Output $\gamma_i$ and ${\bf v}_i$. 
\end{algorithmic}
\end{algorithm}

For the Dirichlet problem, we approximate the solutions using those of the Robin problem with small $\beta$. 
The algorithms for solving the Dirichlet problem and the corresponding eigenproblems are summarized in 
Algorithm 4 PoissonDirichlet and Algorithm 5 EigenDirichlet, respectively. In the following algorithm, 
for any subset $X\subset P$,  use $X$ to also denote the set of indices of the elements in $X$.   

\begin{algorithm}[h!]
\floatname{algorithm}{Algorithm}
\caption{ { PoissonDirichlet($P, S, \bV, \bA, {\bf f}, {\bf g}, t, \beta$)} }
\begin{algorithmic}[1]
   \STATE Compute the matrices $\mathcal{L}, \mathcal{I}, \mathcal{B}$. 
	\STATE Set ${\bf b } = \frac{2}{\beta} \mathcal{B} {\bf g} + \mathcal{I} {\bf f}$. 
	\STATE Set $K = \mathcal{L}$ and modify $K(P, S) = K(P, S) + \frac{2}{\beta}\mathcal{B}$
	\STATE Solve the linear system $K {\bf u} = {\bf b}$ and obtain ${\bf u} = (u_1, \cdots, u_n)$.
	\STATE Output ${\bf u}$. 
\end{algorithmic}
\end{algorithm}

\begin{algorithm}[h!]
\floatname{algorithm}{Algorithm}
\caption{ { EigenDirichlet($P, S, \bV, \bA, t, \beta$)} }
\begin{algorithmic}[1]
   \STATE Compute the matrices $\mathcal{L}, \mathcal{I}, \mathcal{B}$. 
	\STATE Set $K = \mathcal{L}$ and modify $K(P, S) = K(P, S) + \frac{2}{\beta}\mathcal{B}$
	\STATE Solve the generalized eigenproblem $K{\bf v} =  \gamma \mathcal{I}{\bf v}$ 
and obtain the eigenvalues
	$0<\gamma_1\leq \gamma_2 \leq \gamma_3, \cdots$ and the corresponding eigenvectors ${\bf v}_0, ~{\bf v}_1, ~{\bf v}_2, \cdots$
	\STATE Output $\gamma_i$ and ${\bf v}_i$. 
\end{algorithmic}
\end{algorithm}
Note that the choice of $\beta$ in the above two algorithms has to be small to achieve a 
good approximation. On the other hand, it can not be too small and is theoretically 
at least of order $\sqrt{t}$ (see Theorem $3.3$~\cite{SS14}) for ${\bf u}$ 
computed by the algorithm PoissonDirichlet to converge. 

In all the above algorithms, there is a parameter $t$, whose choice depends on the input data, 
in particular, the density of $P$ and $S$. In Section~\ref{sec:examples}, we will show 
how to empirically choose $t$ 
to achieve the best accuracy. For the choice of $t$ with theoretically guaranteed convergence, 
the readers are referred to~\cite{SS14}. 

Finally, we write down the ALM iterative algorithm for solving the Dirichlet problem 
given in~\eqref{eq-dirichlet}. Recall in this method, the Dirichlet problem is modeled 
as a constrained optimization problem and is solved by an ALM iterative procedure where 
each iteration consists of solving a Robin problem as given in~\eqref{eq-robin-ALM}. 
See the algorithm ALMDirichlet. 
The purpose of the ALM iteration is to alleviate the requirement for $\beta$ being small, 
which is demonstrated empirically in Section~\ref{sec:examples}. We remark that we have not
been able to show this ALM iterative procedure converge theoretically, although we have provided
some evidence to support its convergence in Section~\ref{convergence-alm}. 

\begin{algorithm}[h!]
\floatname{algorithm}{Algorithm}
\caption{ { ALMDirichlet($P, S, \bV, \bA, {\bf f}, {\bf g}, t, \beta$)} }
\begin{algorithmic}[1]
	\STATE Compute the matrices $\mathcal{L}, \mathcal{I}, \mathcal{B}$. 
	\STATE Set ${\bf b}_0 = \mathcal{I}{\bf f}$. 
	\STATE Set $K = \mathcal{L}$ and modify $K(P, S) = K(P, S) + \frac{2}{\beta} \mathcal{B}$. 
	\STATE Set $\bw = 0$. 
	\REPEAT 
	\STATE Set ${\bf b} = {\bf b}_0 + 2\mathcal{B}(\frac{1}{\beta}{\bf g} + \by)$.
	\STATE Solve the linear system $K {\bf u} = {\bf b}$ and obtain ${\bf u} = (u_1, \cdots, u_n)$.
	\STATE Modify $\bw=\bw + \frac{1}{\beta} ({\bf g} - {\bf u}(S))$.
	\UNTIL { $\| {\bf g} - {\bf u}(S)\| $ = 0 }.
	\STATE Output ${\bf u}$. 
\end{algorithmic}
\end{algorithm}

\section{Volume weight estimations}
\label{sec:weight}
In this section, for the sake of completeness, we give a brief description of the approach proposed in~\cite{LuoSW09} 
to estimate the volume weight vector $\bV$ from the point sets $P$. Using the same approach, the weight vector $\bA$ can 
be estimated from $S$. 
The basic idea is to construct a local patch around a sample point, from which the weight of that point is computed. 
The detailed algorithm is described in Algorithm $7$ EstimateWeights. 
\begin{algorithm}[h!]
\floatname{algorithm}{Algorithm}
\caption{ { EstimateWeights($P, k, n$)} }
\begin{algorithmic}[1]
	\FOR {each point $p\in P$ }
	\STATE Find the $n$-nearest neighbors of $p$ in $P$, denoted $N_p$.  
	\STATE Set $\delta = \frac{1}{|N_p|} \sum_{q\in N_p} \|p-q\|$ and $N_{\delta} = \{q\in P | \|p-q\| < \delta \}$ be the 
	points in $P$ within $\delta$ distance to $p$.
	\STATE Estimate the tangent space at $p$ by a $k$-dimensional subspace $\tilde{T}_p$ estimated from $N_{\delta}$ using weighted least squares.  
	\STATE Project the points in $N_{\delta}$ into $\tilde{T}_p$ and denote them by $\tilde{N}_{\delta}$
	\STATE Compute the Voronoi diagram of $\tilde{N}_{\delta}$ on $\tilde{T}_p$. 
	\STATE The volume weight $V_p$ is estimated as the volume of the Voronoi cell.
	\ENDFOR
\end{algorithmic}
\end{algorithm}

Theoretically, if $\delta$ in Algorithm $7$ is fixed to be a fraction of the reach of $\M$, 
then we have the following theorem which guarantees that the integral of any Lipschitz function 
on $\M$ can be well approximated using the volume weights $V_p$ estimated by Algorithm $7$. 
A sampling $P$ of $\M$ is an $(\e, \eta)$-sampling if for any point $x\in M$, there is a point $p\in P$, 
so that $|x-p|<\e$ and for any two different sample points $p, q\in P$, $|p-q|>\eta$. 
\begin{theorem}[\cite{LuoSW09}]
Given an $(\e, \eta)$-sampling $P$ of $\M$ with $\e$ sufficiently small,
compute the volume weight $V_p$ for each $p\in P$ using Algorithm $7$. Then for any Lipschitz function
$f$  we have that
$$
\left| \int_\M f - \sum_{p \in P} V_p f(p) \right| =~O(\e + \e^3 / \eta^2),
$$
implying that for $\eta = \Omega(\e^{3/2 - \xi})$ with any positive constant $\xi$, we have
$$
\lim_{\e \rightarrow 0} \left| \int_\M f - \sum_{p \in P} V_p f(p) \right| =~ 0.
$$
\end{theorem}

The reach of $\M$ is usually unknown. In practice, $\delta$ is estimated using the average distance to 
the $n$-nearest neighbors as described in Algorithm $7$, which works well. If $\M$ has boundary, 
for a point $p$ near to the boundary, we take as the volume weight $V_p$ the volume of the Voronoi cell
which is inside the Convex hull of $\tilde{N}_{\delta}$. 

\section{Numerical Results}
\label{sec:examples}
In this section, we run our point integral method on several examples, mainly to explain the choice of 
the parameters in our algorithm, to demonstrate the convergence of our algorithm and to compare 
its performance with that of finite element method. The approximation error is computed in $L_2$: 
$err = \|u - u_{gt}\|/\|u_{gt}\|$ where the $L_2$ norm is evaluated 
as $\|f\| = \sqrt{\sum_{\bfp_i\in P} f^2_iV_i}$ for a function $f$ over $\M$
and $\|f\| = \sqrt{\sum_{\bfs_i\in S} f^2_iA_i}$ for a function $f$ over $\p\M$.

\subsection{Unit Disk}
Test our algorithm on unit disk. We discretize unit disk using a Delaunay mesh with 
$684$ vertices shown in Figure~\ref{fig:unit_disk}(a). This mesh is generated using
Triangle~\cite{shewchuk96b}. We obtain a sequence of refined meshes with
$2610$, $10191$ and $40269$ vertices by subdividing it once, twice and three times. 
In each subdivision, a triangle in the mesh is split into four smaller ones using the 
midpoints of the edges. Note that  the mesh size is reduced by half but 
the number of vertices roughly get quadrupled for each subdivision. 
Figure~\ref{fig:unit_disk}(b) shows the mesh after one subdivision.  
For point integral method, we remove the mesh topology and only retain the vertices as 
the input point set $P$. Those vertices on the boundary of the mesh are taken as the 
input point set $S$.

\begin{figure}[!t]
\begin{center}
\begin{tabular}{cc}
\includegraphics[width=0.45\textwidth]{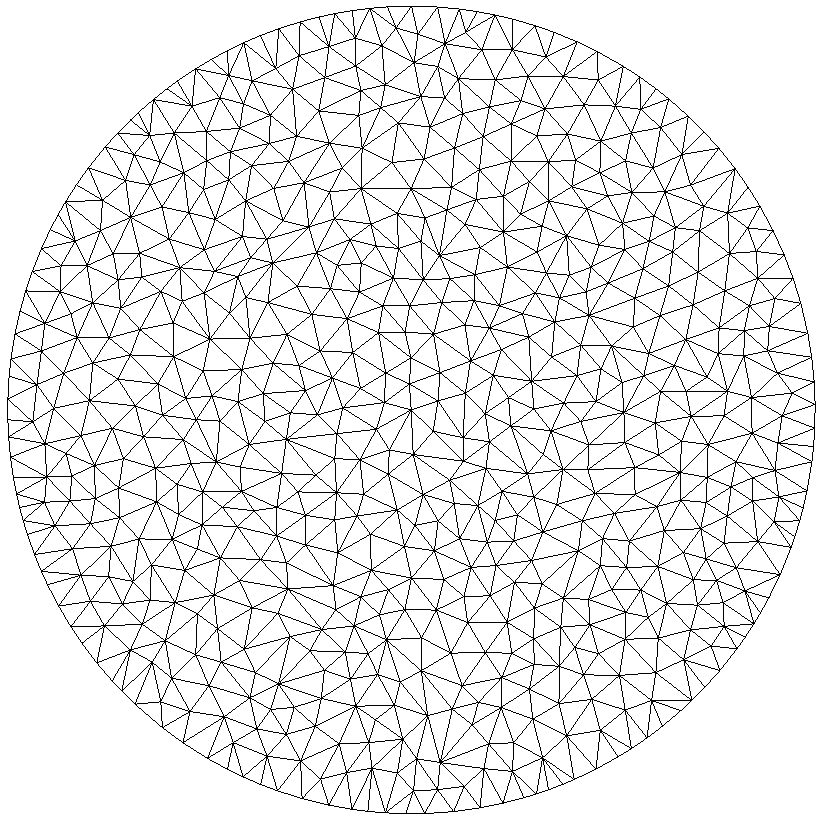} & \includegraphics[width=0.45\textwidth]{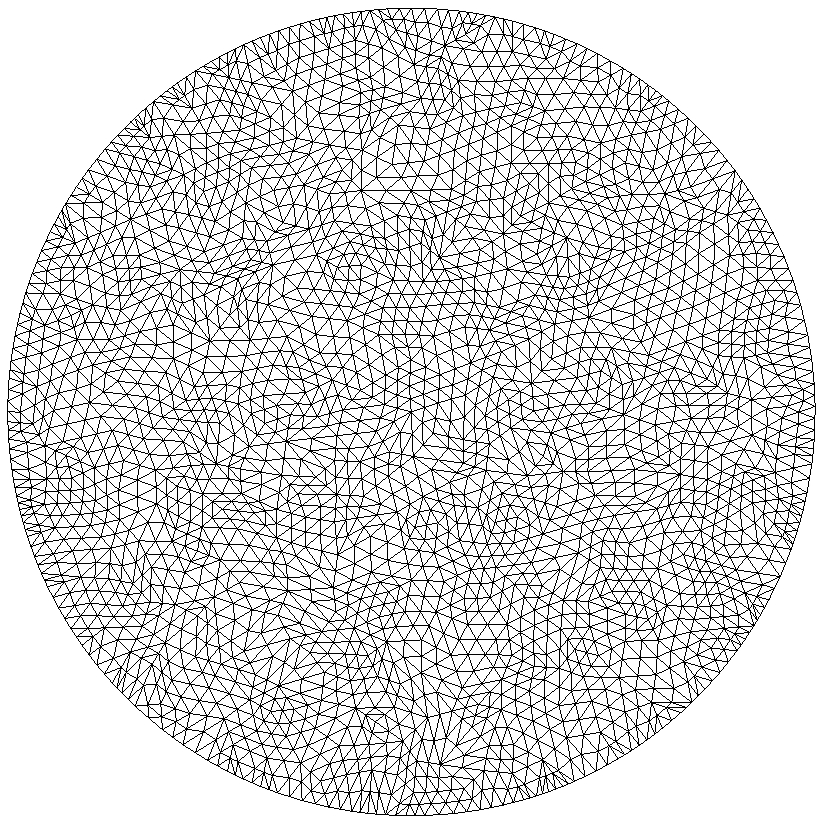} \\
(a) & (b)
\end{tabular}
\end{center}
\vspace{-4mm}
\caption{Discretization of unit disk. (a) A triangle mesh of unit disk with $684$ vertices. (b) The mesh
with $2610$ vertices obtained by subdividing the triangle mesh in (a) where each triangle is subdivided into four 
using the midpoints of the edges. }
\label{fig:unit_disk}
\end{figure}

\vspace{0.1in}
\noindent{\bf Choice of Parameters:~}
Our algorithm has two parameters $t$ and $\beta$. Here we show how the approximation error changes with 
different choices of $t$ and $\beta$.  
Set the boundary condition (both Neumann and Dirichlet) as that of the function $u_{gt} = \cos 2 \pi r$ with
$r = \sqrt{x^2 + y^2}$ and see how accurate our algorithm can recover this function.

Figure~\ref{fig:unit_disk_param_t} shows the plot of the approximation error $\|u-u_{gt}\| / \|u_{gt}\|$
as a function of the parameter $\sqrt{t}$. The approximating solution $u$ is computed
by Algorithm $2$ for the Neumann boundary and by Algorithm $6$ for the Dirichlet boundary. 
In Algorithm $6$, set $\beta = 1$ and the solution is obtained after $100$ iterations.
Given a sampling $P$ on $\M$, let $\delta_i$ be the average distance 
from $p_i\in P$ to its $10$ nearest neighbors in $P$ and $\delta$ is the average of $\delta_i$ 
over all points $p_i\in P$. 
We observe, from the plots in Figure~\ref{fig:unit_disk_param_t},  that the 
optimal parameter $\sqrt{t}$ which produces the smallest approximation error remains 
$0.5\delta$ for the Neumann boundary and $0.75\delta$ for the Dirichlet boundary across the above 
sequence of refined samplings. 
This means only a fixed number of samples are empirically needed in the neighborhood of size $\sqrt{t}$
for PIM to converge. Such choice of parameter $t$ leads to a better empirical convergence rate
than what is predicted in~\cite{SS14}.
The theoretical analysis for PIM in the paper~\cite{SS14} shows that
the convergence of PIM requires more and more samples in the neighborhood of size $\sqrt{t}$
as $t$ decreases, and in fact requires infinitely many in the limit of $t$ going to $0$.
As we will see below, PIM empirically 
converges at least linearly in mesh size, while our analysis in~\cite{SS14} shows that the convergence rate 
is one fifth root of mesh size. This phenomenon is also observed on 3D domain, as we will show
in Section~\ref{unit-ball}. This suggests that there may be rooms
to improve our analysis on the convergence rate. 


\begin{figure}[!t]
\begin{center}
\begin{tabular}{cc}
\includegraphics[width=0.48\textwidth]{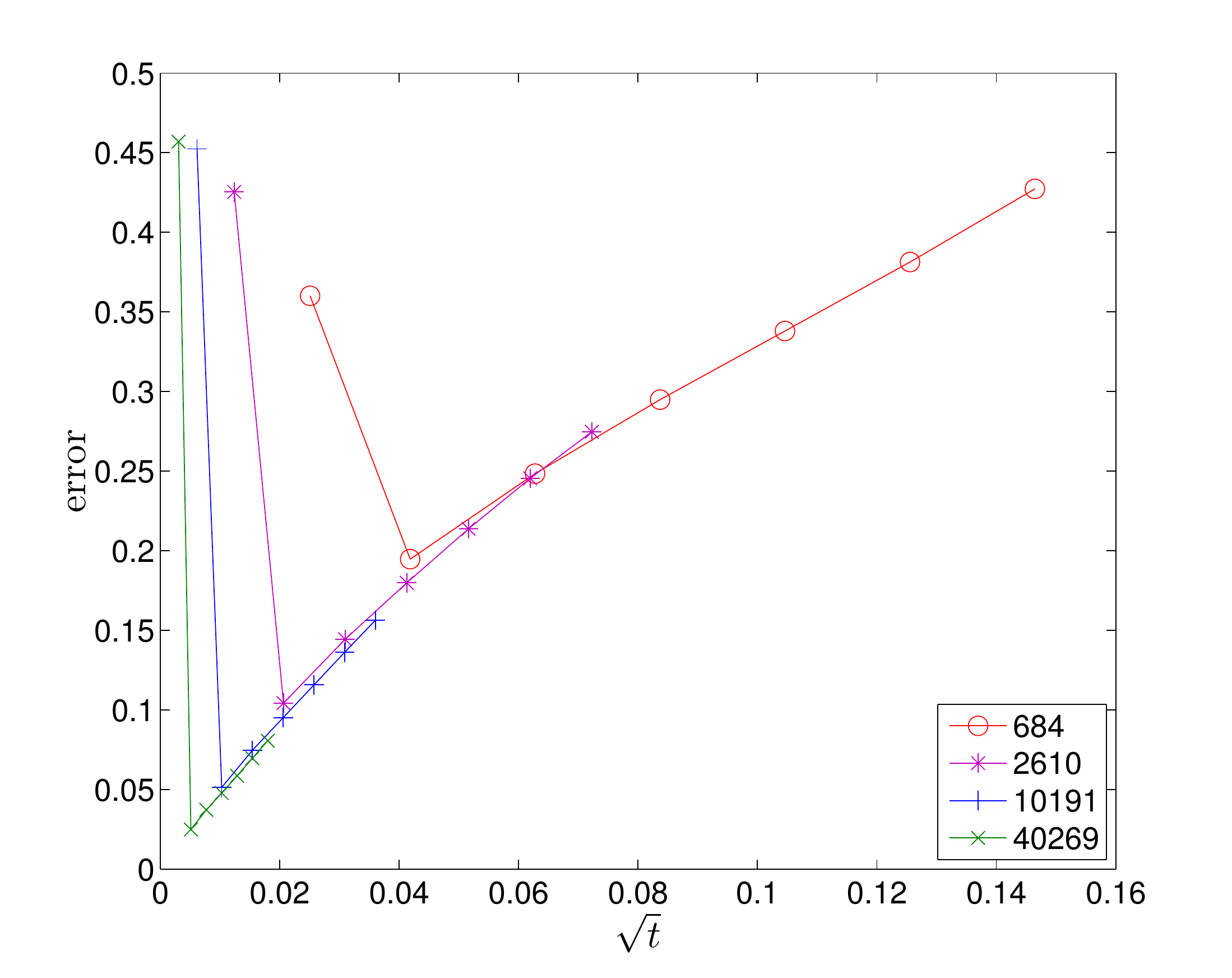}
&
\includegraphics[width=0.48\textwidth]{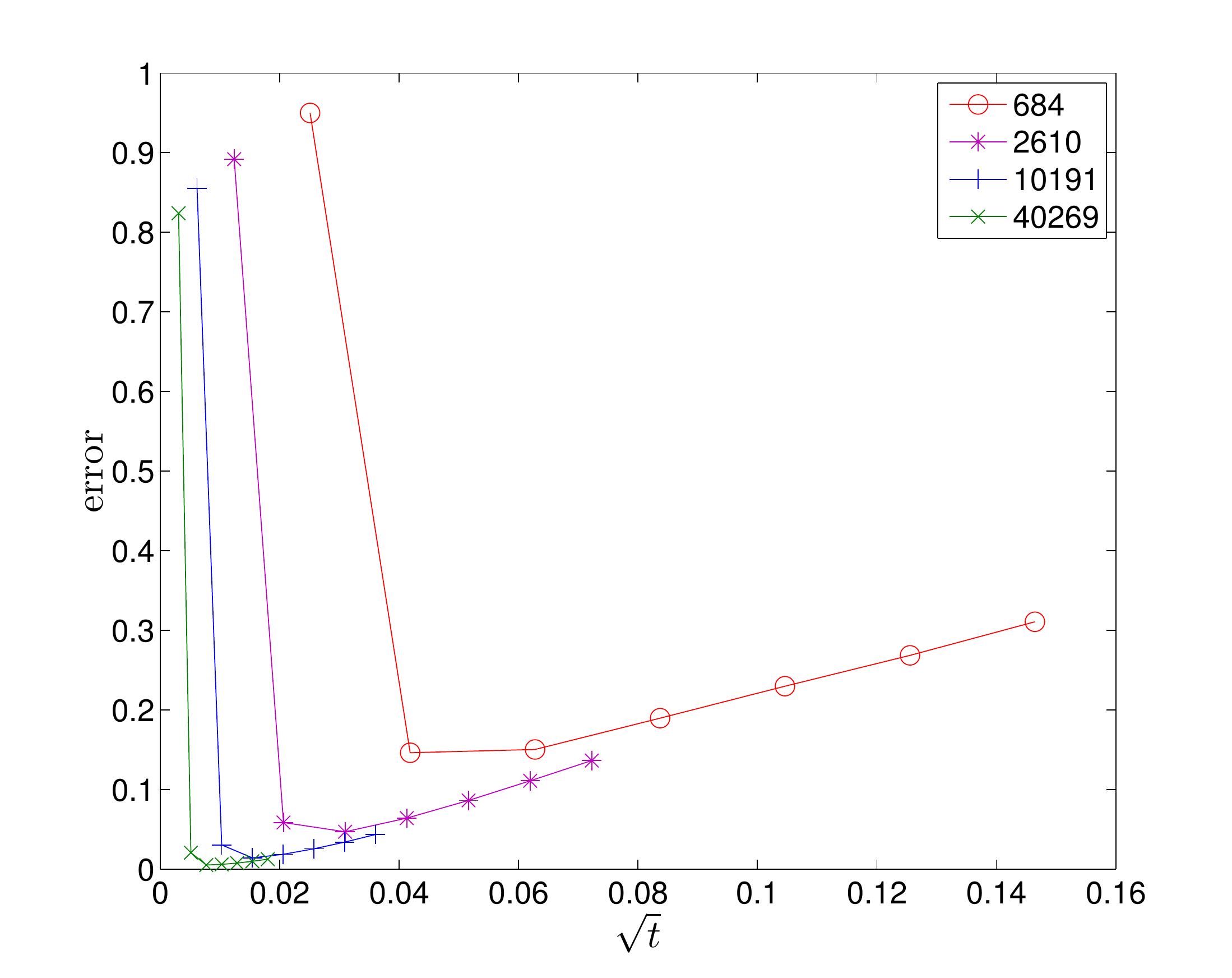}
\\
(a) & (b)
\end{tabular}
\end{center}
\vspace{-4mm}
\caption{Approximation error vs. the parameter $t$ on unit disk: (a) Neumann boundary; (b) Dirichlet boundary}
\label{fig:unit_disk_param_t}
\end{figure}


To see the choice of the parameter $\beta$, we fix the parameter $\sqrt{t} = 0.75\delta$.
we first show how the choice of $\beta$ affects Algorithm $4$. 
Figure~\ref{fig:unit_disk_param_beta_penalty} shows the approximation errors for the solution
computed by Algorithm $4$ using different $\beta$ over the above sequence of refined samplings.  
As we can see, the effect of $\beta$ is similar across different samplings: the approximation errors remain small 
for $\beta$ in the interval $[10^{-6}, 10^{-3}]$ but increases significantly as $\beta$ increases from $10^{-3}$
or decreases from $10^{-6}$. 
This phenomenon fits the theory of PIM~\cite{SS14} well: The smaller the $\beta$ is, 
the smaller the approximation error is; and on the other hand, if $\beta$ is chosen too small, 
the linear system becomes numerically unstable and the approximation error increases.
For a technical reason, our analysis in~\cite{SS14} also requires that $\beta$ and $\sqrt{t}$ are of the same order.
However, it seems not necessary in our experiments, which means we may improve the analysis to 
remove this extra requirement. In the following experiments, we fix $\beta = 10^{-4}$ in Algorithm $4$. 

\begin{figure}[!h]
\begin{center}
\begin{tabular}{c}
\includegraphics[width=0.6\textwidth]{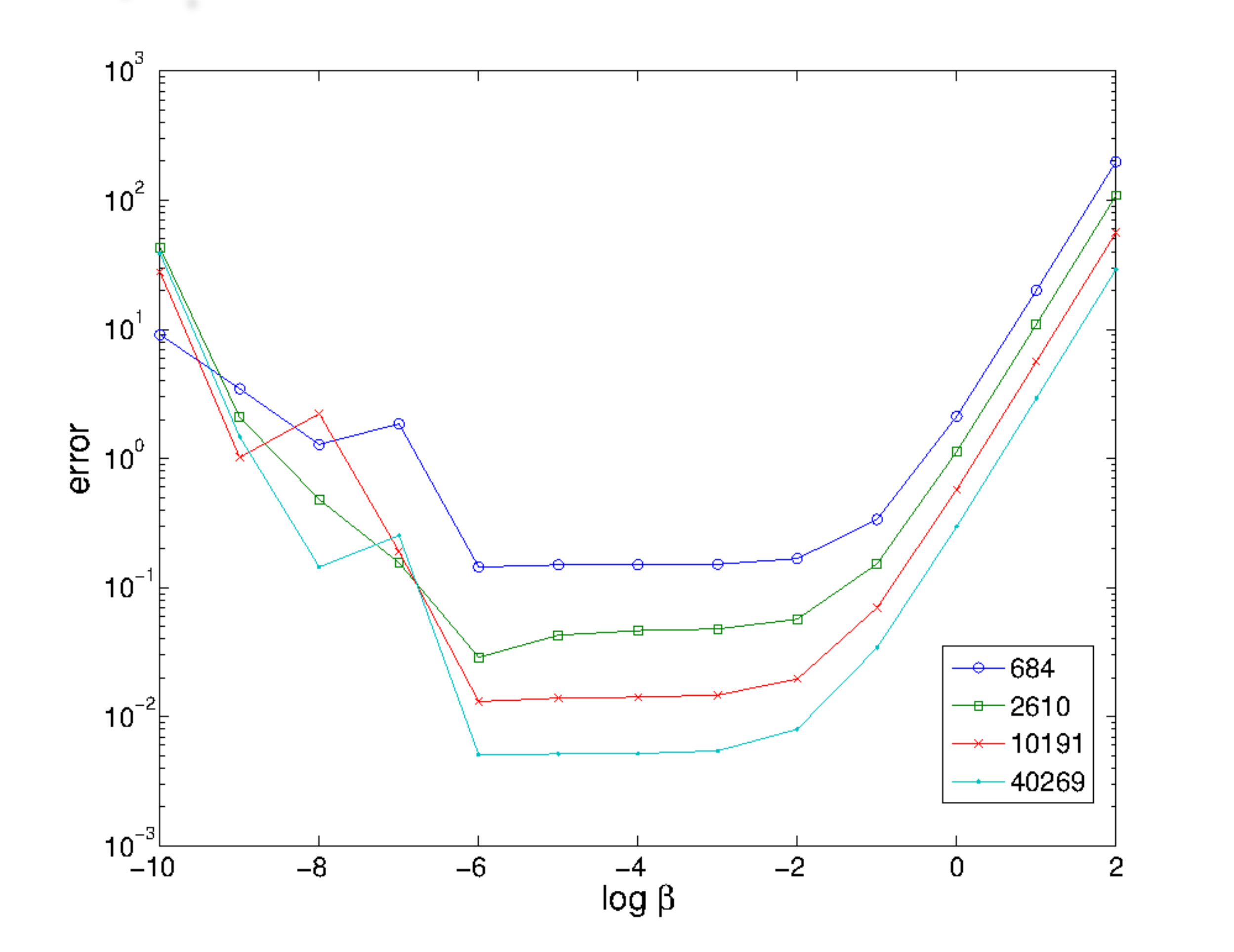}
\end{tabular}
\end{center}
\vspace{-4mm}
\caption{Approximation error vs. parameter $\beta$ by Algorithm $4$ on unit disk.}
\label{fig:unit_disk_param_beta_penalty}
\end{figure}

Next we show how the choice of $\beta$ affects Algorithm $6$, which employs the approach of 
augmented Lagrangian multiplier. 
We run the algorithm over the sampling with $2610$ points. Recall when the ALM iteration converges, 
the obtained solution should satisfy the specified boundary 
condition.  
Assume $v^k$ is the solution obtained after $k$th iteration. 
Figure~\ref{fig:unit_disk_param_beta}(a) shows the approximation error $\|v^k|_{\p\M} - g\|/\|g\|$ 
on the boundary. As we can see, the smaller the parameter $\beta$ is, the faster the solution $v^k$ 
converges on the boundary. However the algorithm diverges if $\beta$ is too small (less than $5\times 10^{-6}$). 
Nevertheless, the solution converges on the boundary over a large range of $\beta$. 
Figure~\ref{fig:unit_disk_param_beta}(b) shows the approximation errors $\|u - u_{gt}\|/\|u_{gt}\|$ 
after $100$ iterations.  As we can see, although the algorithm converges at the different speeds for 
the different $\beta$, the difference in the final approximation errors is small across the different 
but reasonable choices of $\beta$. Thus Algorithm 6 which employs ALM iteration is not sensitive 
to the choice of $\beta$ and works over a large range of $\beta$. 

\begin{figure}[!t]
\begin{center}
\begin{tabular}{cc}
\includegraphics[width=0.48\textwidth]{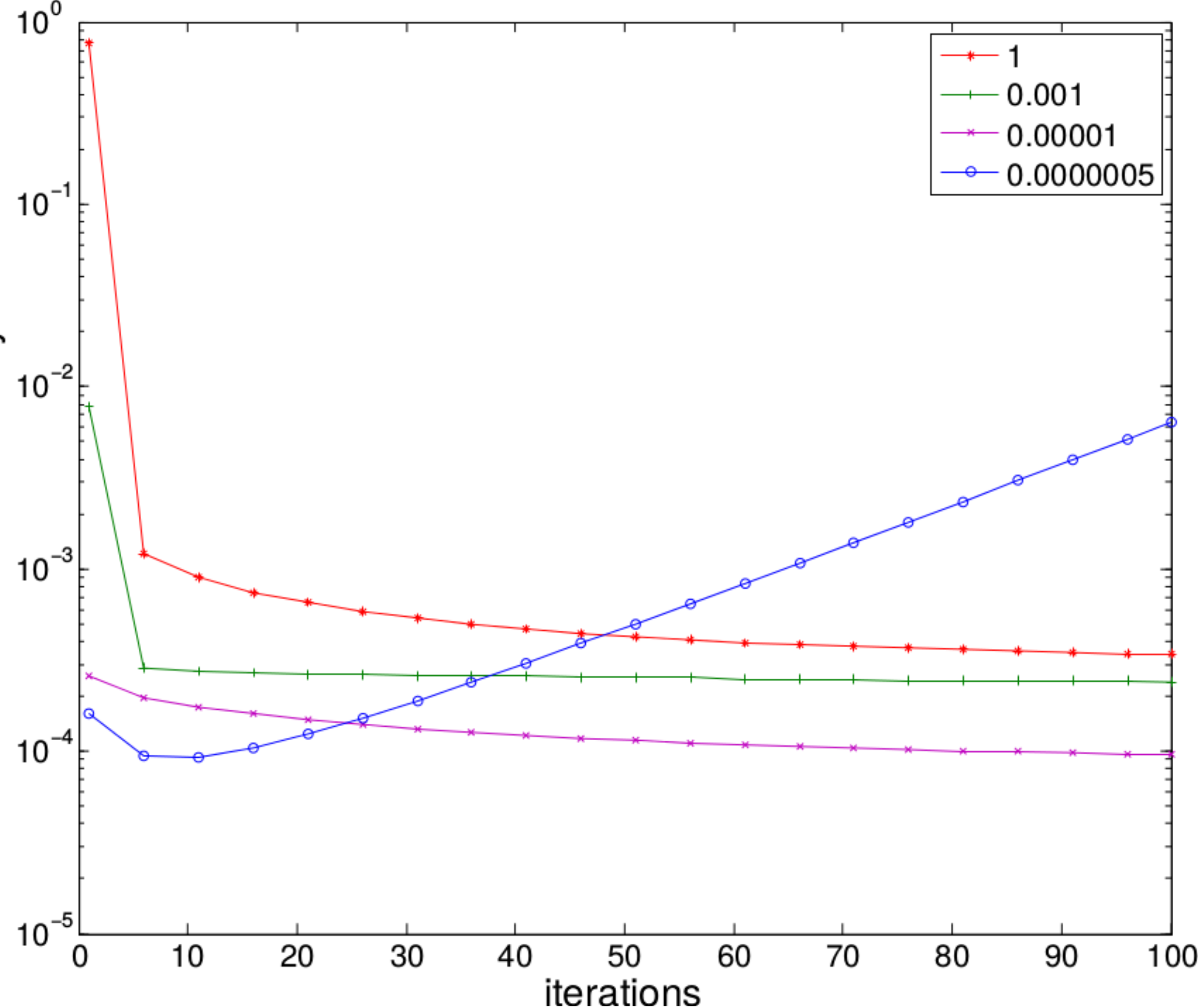}
&
\includegraphics[width=0.48\textwidth]{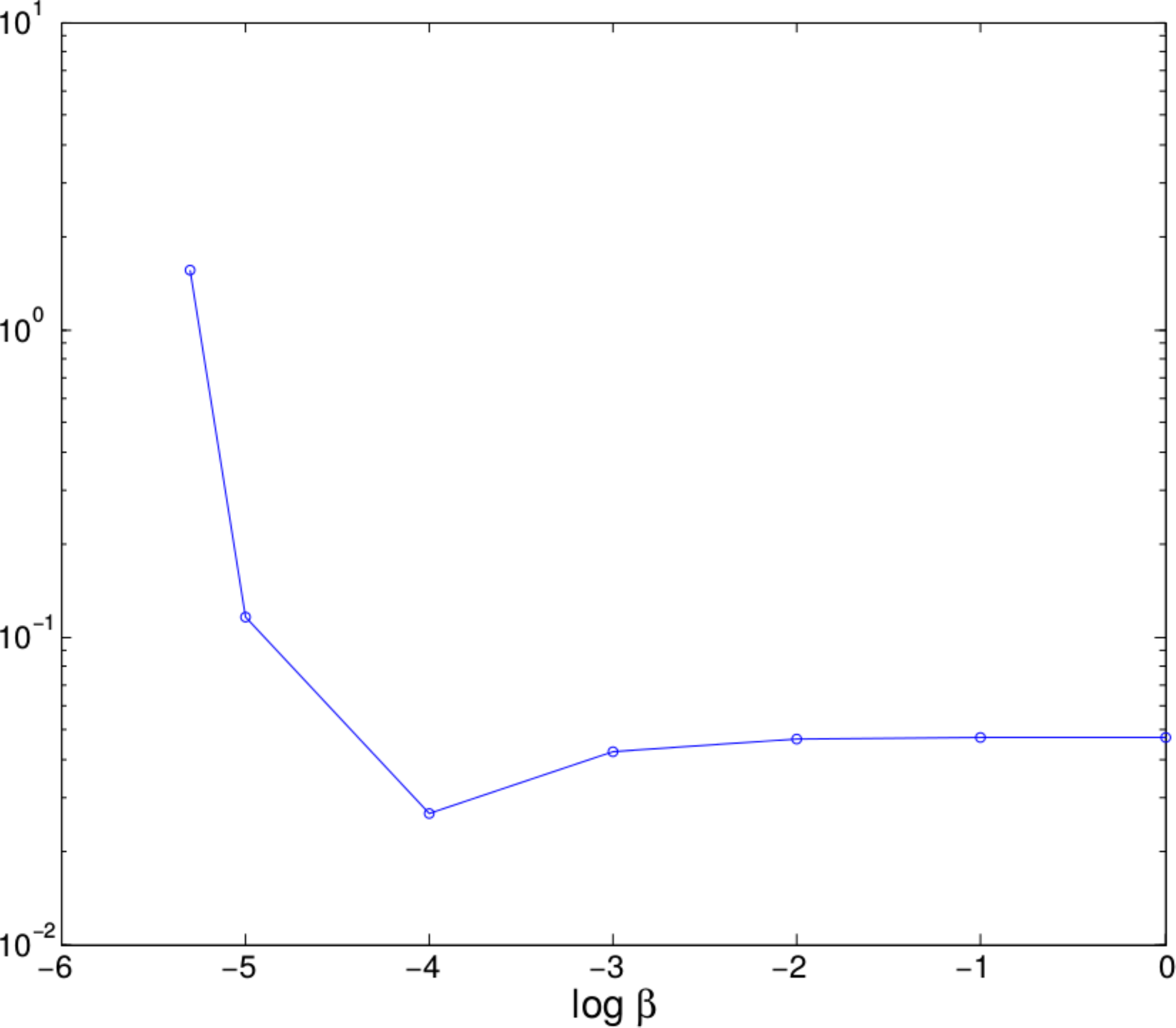}
\\
(a) & (b)
\end{tabular}
\end{center}
\vspace{-4mm}
\caption{Choice of parameter $\beta$: (a) Convergence of $v^k$ on the boundary under different $\beta$; (b) Approximation error vs. parameter $\beta$. }
\label{fig:unit_disk_param_beta}
\end{figure}

\vspace{0.1in}
\noindent {\bf Convergence for the Poisson Equation:~}
We fix $\sqrt t= 0.75\delta$ and $\beta = 10^{-4}$. We show the convergence of Algorithm 2 and Algorithm 4
for the Neumann boundary and the Dirichlet boundary respectively, and also compare them to the results of FEM. 
In FEM, we use linear elements. 
Table~\ref{tbl:unit_disk_le} shows the approximation
error for recovering the function  $\cos 2 \pi r$ over a sequence of refined meshes or samplings. 
As we can see, FEM has the quadratic convergence
rate for both the Neumann boundary and the Dirichlet boundary, which coincides with the theory of FEM. 
PIM converges in the linear order $h$ for the Neumann boundary  and 
in the order $h^{3/2}$ for the Dirichlet boundary, 
where $h$ is referred to mesh size. This convergence rate is much 
faster than the order $h^{1/5}$  predicted by our analysis of PIM in~\cite{SS14}. 

\begin{table}[!ht]
\begin{center}
\begin{tabular}{| c| c | c | c | c |}
\hline
$|V|$	& 684 	&  2610 & 10191 & 40296 \\ 
\hline
\multicolumn{5}{|c|}{Neumann Boundary}\\
\hline
FEM     & 0.0212 & 0.0056 & 0.0014 & 0.00036\\
\hline
PIM       & 0.1947 & 0.1043 & 0.0513 & 0.0249\\
\hline
\multicolumn{5}{|c|}{Dirichlet Boundary}\\
\hline
FEM     & 0.0310 & 0.0079 & 0.0020 & 0.0005\\
\hline
PIM      &  0.1500 &   0.0428  &  0.0140  &  0.0052\\
\hline
\end{tabular}
\end{center}
\caption{Convergence for recovering the function $\cos 2\pi  r$. 
The solution is computed using Algorithm 2 for Neumann 
boundary and Algorithm 4 for Dirichlet boundary.
\label{tbl:unit_disk_le}}
\end{table}

\begin{figure}[!h]
\begin{center}
\begin{tabular}{c}
\includegraphics[width=0.45\textwidth]{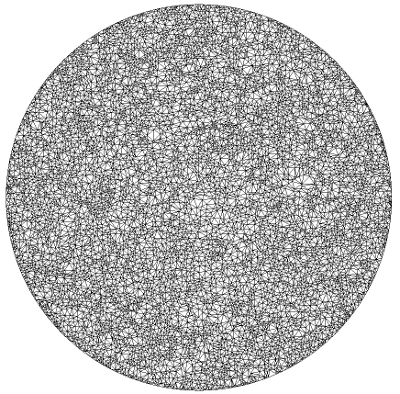} 
\end{tabular}
\end{center}
\vspace{-4mm}
\caption{A triangle mesh of unit disk. }
\label{fig:unit_disk_10000}
\end{figure}

However, although FEM has a better convergence rate, it is sensitive to the quality of the mesh. 
Figure~\ref{fig:unit_disk_10000} shows a Delaunay triangle mesh with $10000$ vertices randomly sampled
on unit disk. The condition number of the stiff matrix of FEM reaches $10^{20}$. 
Table~\ref{tbl:unit_disk_bad_triangle} shows the approximation errors for recovering the function $\cos 2\pi r$ 
and the function $x^2-y^2$. As we can see, FEM is not stable and may produce
a solution with no accuracy. However, PIM always produces a solution with reasonable accuracy.

\begin{table}[!h]
\begin{center}
\begin{tabular}{| c| c | c |}
\hline
		& Neumann Boundary    & Dirichlet Boundary \\
\hline
FEM  & 0.0026 &  2.0218\\
\hline
PIM &  0.0600 & 0.0673 \\
\hline
\multicolumn{3}{|c|}{$\cos 2\pi r$}\\
\hline
FEM  & 1.2003 & 5.0321 \\
\hline
PIM &  0.0610 & 0.0081 \\
\hline
\multicolumn{3}{|c|}{$x^2 - y^2$}\\
\hline
\end{tabular}
\end{center}
\caption{The approximation errors of FEM and PIM in solving the Poisson Equations over the mesh shown in Figure~\ref{fig:unit_disk_10000}
\label{tbl:unit_disk_bad_triangle}}
\end{table}


\begin{figure}[!h]
\begin{center}
\begin{tabular}{cc}
\includegraphics[width=0.45\textwidth]{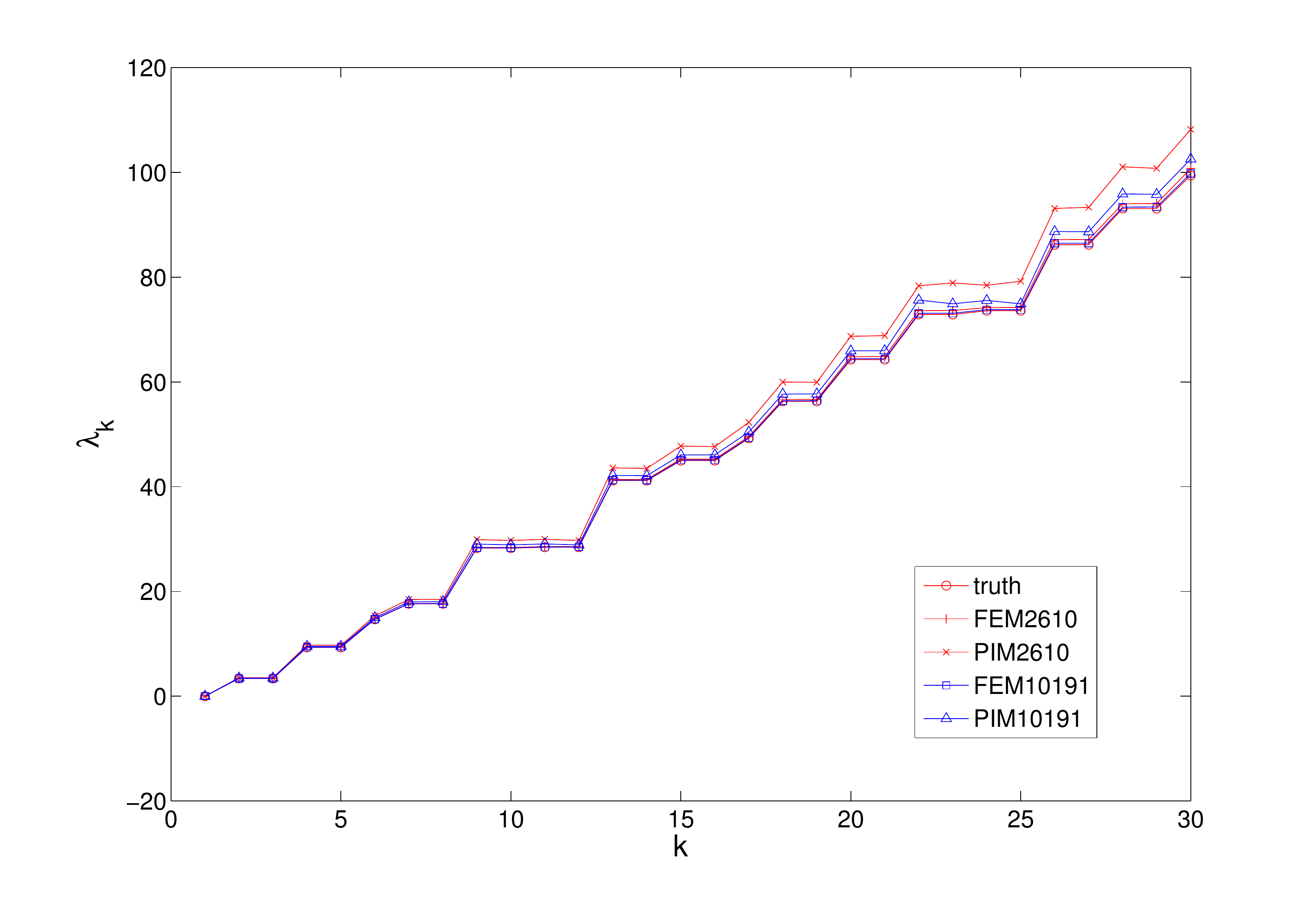}&
\includegraphics[width=0.45\textwidth]{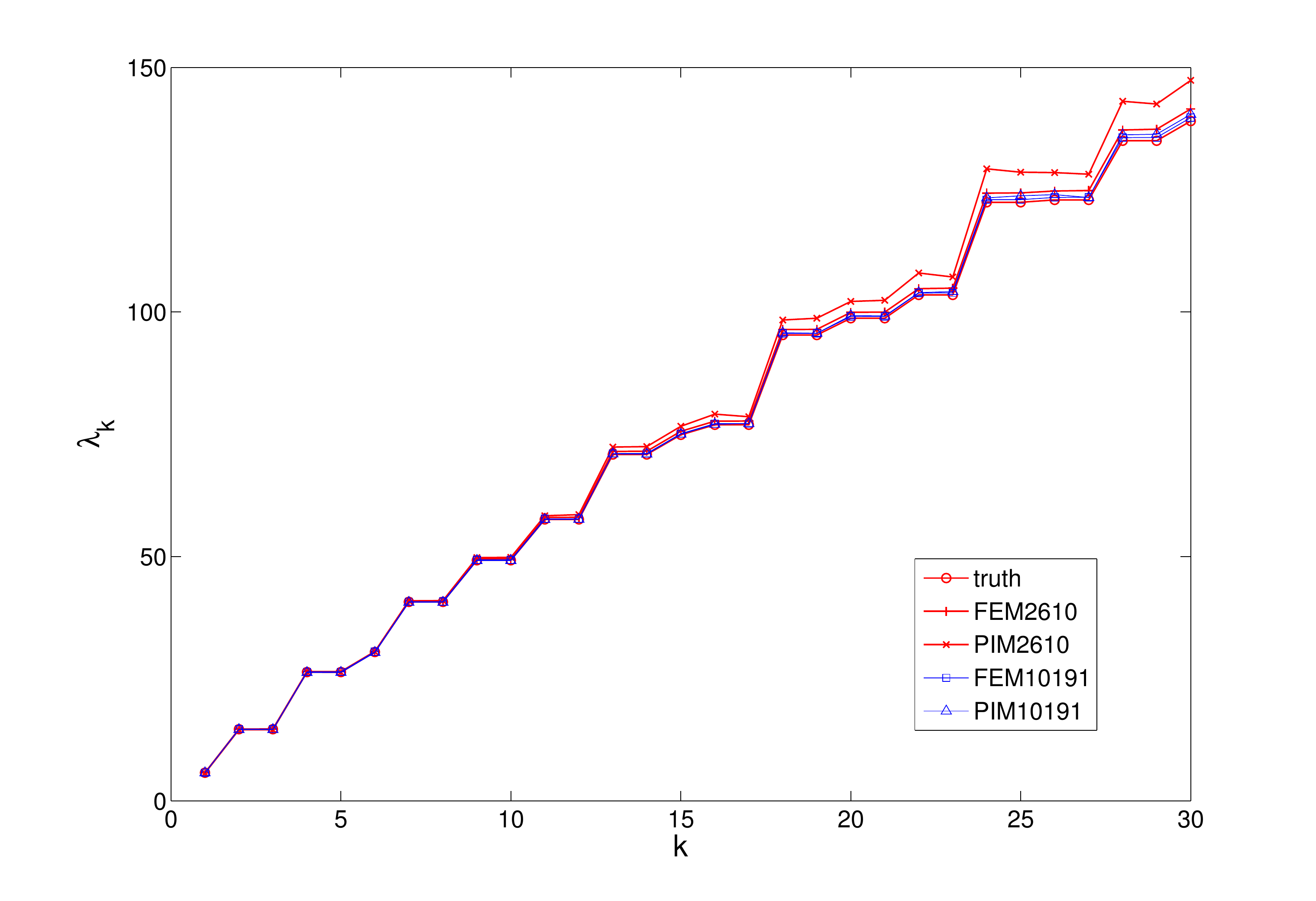}\\
(a)&(b)
\end{tabular}
\end{center}
\vspace{-4mm}
\caption{The eigenvalues of unit disk estimated by FEM and PIM over the meshes or the samplings with
2610 points and 10191 points. 
(a) Neumann eigenvalues; (b) Dirichlet eigenvalues}
\label{fig:unit_disk_eig}
\end{figure}

\begin{figure}[!h]
\begin{center}
\begin{tabular}{cc}
\includegraphics[width=0.45\textwidth]{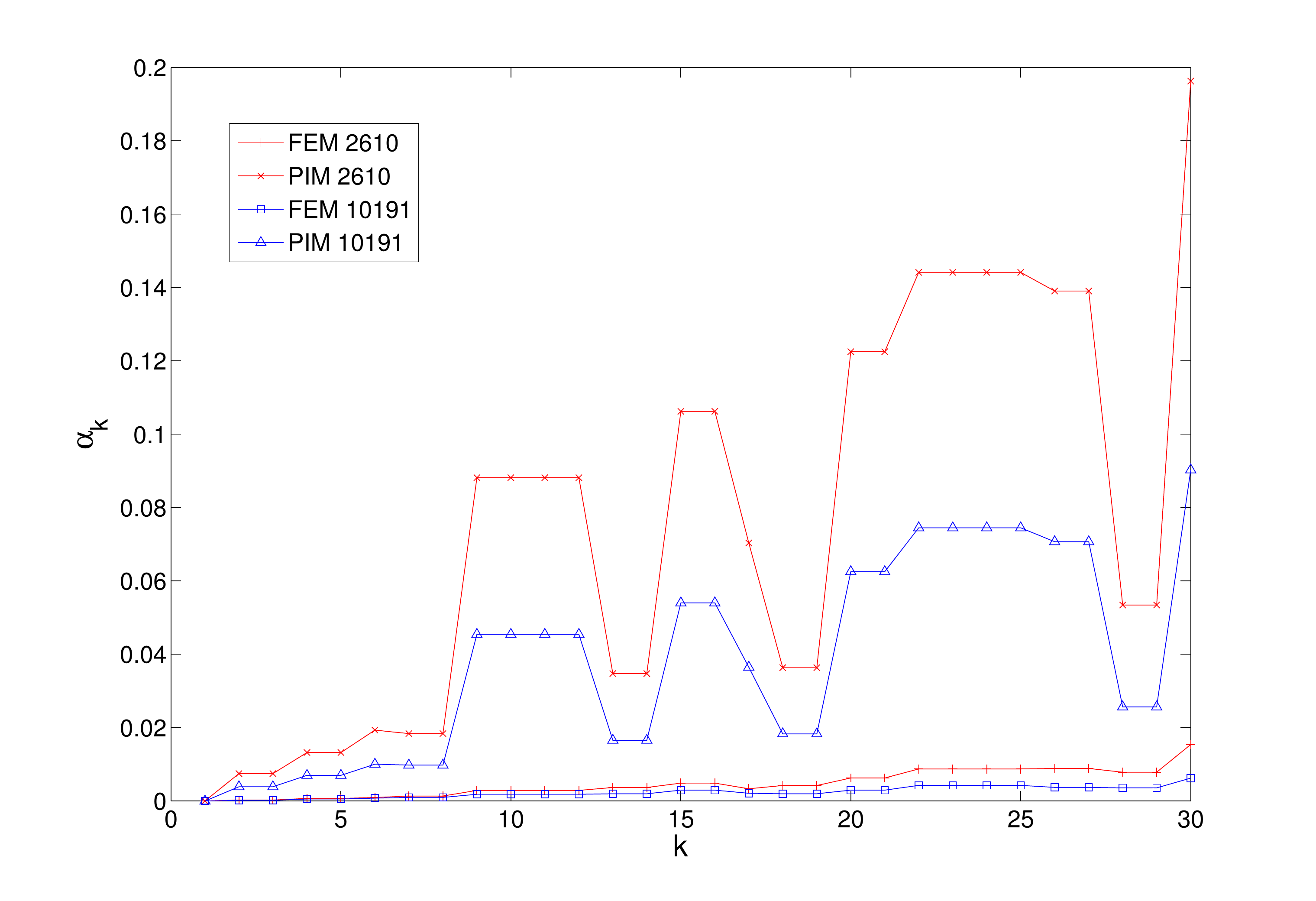}&
\includegraphics[width=0.45\textwidth]{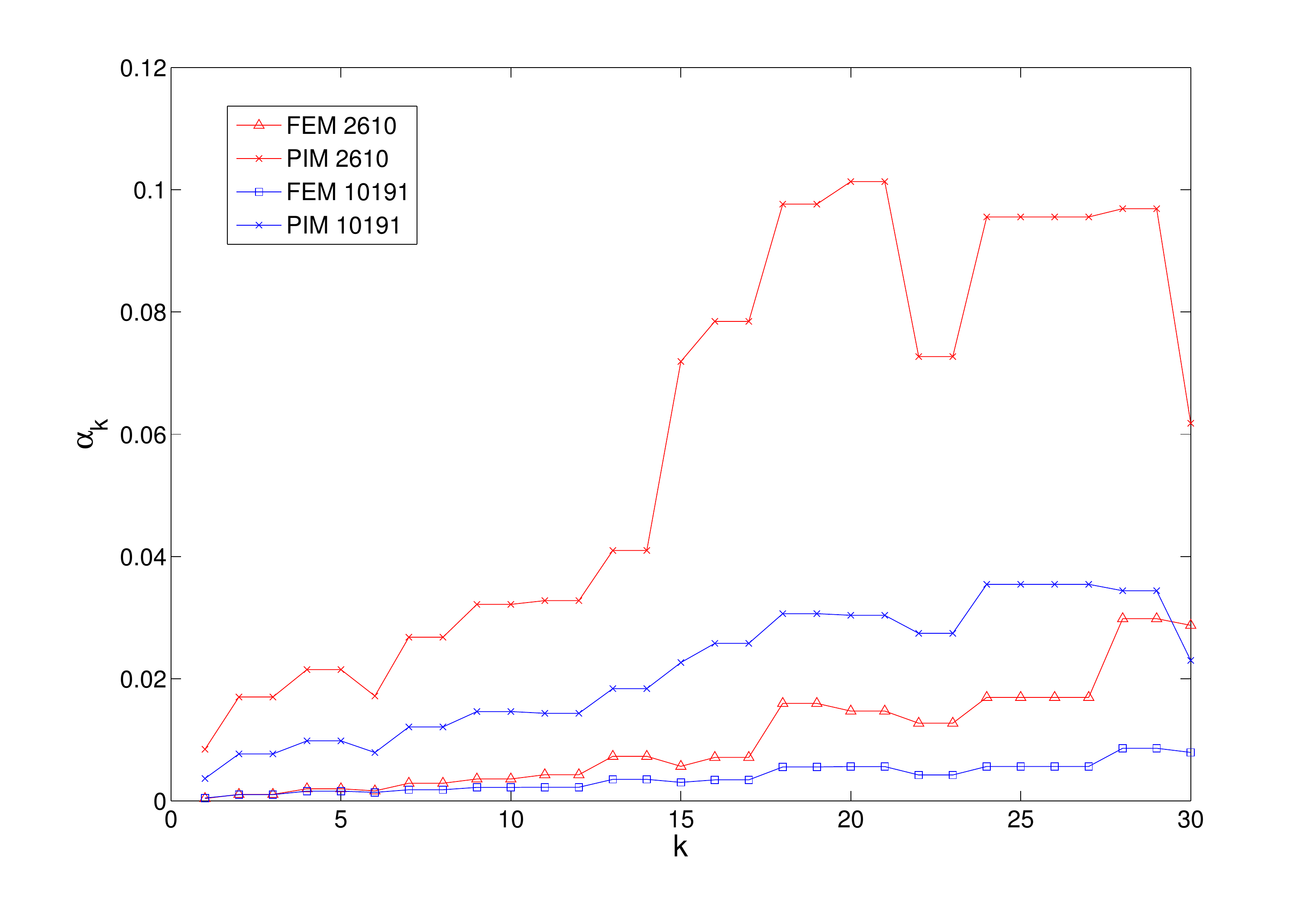}\\
(a)&(b)
\end{tabular}
\end{center}
\vspace{-4mm}
\caption{The approximation errors of the eigenfunctions of unit disk estimated using FEM (PIM) over the meshes (samplings) with 
2610 points and 10191 points. (a) Neumann boundary; (b) Dirichlet boundary.}
\label{fig:unit_disk_eigenfunction}
\end{figure}

\vspace{0.1in}
\noindent{\bf Eigensystem:~} 
We compute the eigensystem of Laplacian using Algorithm 3 for the problem~\eqref{eigen-neumann} with 
the homogeneous Neumann boundary 
and Algorithm 5 for the problem~\eqref{eigen-dirichlet} with the homogeneous Dirichlet boundary. 
Again we fix $\sqrt t=0.75\delta$ and $\beta = 10^{-4}$.

Figure~\ref{fig:unit_disk_eig} shows the first $30$ eigenvalues computed using PIM (FEM) over 
the sampling (mesh) with 2610 points and 10191 points. Both methods give a good estimation for the eigenvalues. 
Figure~\ref{fig:unit_disk_eigenfunction} shows the approximation error of the first 30 eigenfunctions, 
where the approximation error is computed as the angle between the eigenspaces of ground 
truth and the eigenspaces estimated by PIM or FEM. Let $U$ and $V$ be the two subspaces in $\mathbb{R}^{n}$. The angle between
$U$ and $V$ is defined as 
\begin{equation}
\label{eqn:erroreigenfunction}
\cos\angle U, V = \min_{ x\in U, |x| = 1} \max_{y \in V, |y| = 1} x\cdot y .
\end{equation}
It is well-known that when two distinct eigenvalues of a matrix are close to each other, 
their eigenvectors computed numerically can be switched.
Thus, when 
we estimate the approximation error of the eigenfunctions, we merge the eigenspaces of two eigenvalues close to each
other. In Figure~\ref{fig:unit_disk_eigenfunction} (a), 
we merge the eigenspace of the 9th (or 10th) eigenvalue with that of the 11th eigenvalue, 
and the eigenspace of the 22nd (or 23rd) eigenvalue with that of the 24the eigenvalue.
In Figure~\ref{fig:unit_disk_eigenfunction} (a), we merge
the eigenspace of the 24th (or 25th) eigenvalue and that of the 26 eigenvalue. 


Table~\ref{tbl:unit_disk_eig_N} and Table~\ref{tbl:unit_disk_eig_D}
shows the convergence of the 6th eigenvalue and the corresponding eigenfunction 
computed using FEM and PIM. The approximation error of the ith eigenvalue is estimated as 
$|\lambda_i - \lambda_i^{gt}|$ where $\lambda_i^{gt}$ is the ground truth and 
$\lambda_i$ is the numerical estimation. The approximation error of the eigenfunction
is estimated as the angle between two subspaces.

\begin{table}[!h]
\begin{center}
\begin{tabular}{| c| c | c | c | c |}
\hline
$|V|$ & 684 &  2610 & 10191 & 40296 \\ 
\hline
\multicolumn{5}{|c|}{Eigenvalue}\\
\hline
FEM     & 0.1164 & 0.0371 & 0.0168 & 0.0116\\
\hline
PIM       & 0.8244 & 0.2570 & 0.0555 & 0.0212\\
\hline
\multicolumn{5}{|c|}{Eigenfunction}\\
\hline
FEM     & 0.0026 & 0.0009 & 0.0008 & 0.0008\\
\hline
PIM       & 0.0332 & 0.0193 & 0.0100 & 0.0052\\
\hline
\end{tabular}
\end{center}
\caption{Convergence of the Neumann Eigensystem of unit disk. \label{tbl:unit_disk_eig_N}}
\end{table}

\begin{table}[!h]
\begin{center}
\begin{tabular}{| c| c | c | c | c |}
\hline
$|V|$ & 684 &  2610 & 10191 & 40296 \\ 
\hline
\multicolumn{5}{|c|}{Eigenvalue}\\
\hline
FEM     & 0.0126 & 0.0062 & 0.0045 & 0.0041\\
\hline
PIM       & 0.3228 & 0.1778 & 0.1115 & 0.0762\\
\hline
\multicolumn{5}{|c|}{Eigenfunction}\\
\hline
FEM     & 0.0051 & 0.0017 & 0.0014 & 0.0014\\
\hline
PIM       & 0.0313 & 0.0172 & 0.0079 & 0.0034\\
\hline
\end{tabular}
\end{center}
\caption{Convergence of the Dirichlet Eigensystem of unit disk. \label{tbl:unit_disk_eig_D}}
\end{table}

\subsection{Unit Ball}
\label{unit-ball}
Now test PIM on unit ball. The main purpose of this set of experiments is to see
how PIM performs on 3D domains and what are the good ranges of the parameters 
for 3D domains. We discretize unit ball using 3D mesh generation package provided by 
CGAL~\cite{cgal} which is state of the art in mesh generation and uses the approach of Delaunay 
refinement and CVT-type of optimization for improving the mesh quality. 
We obtain a sequence of four refined meshes where the mesh size of a mesh is 
reduced roughly by half from the previous mesh. The number of vertices of the meshes are $546$, 
$3481$, $25606$ and $195725$.  Figure~\ref{fig:unit_ball}(a) and (b) shows the mesh with $546$ and $25606$ vertices
respectively. Similarly, for PIM, we remove the mesh topology and only retain the 
vertices as the input point set $P$. Those vertices on the boundary of the mesh are taken as the 
input point set $S$.

\begin{figure}[!t]
\begin{center}
\begin{tabular}{cc}
\includegraphics[width=0.45\textwidth]{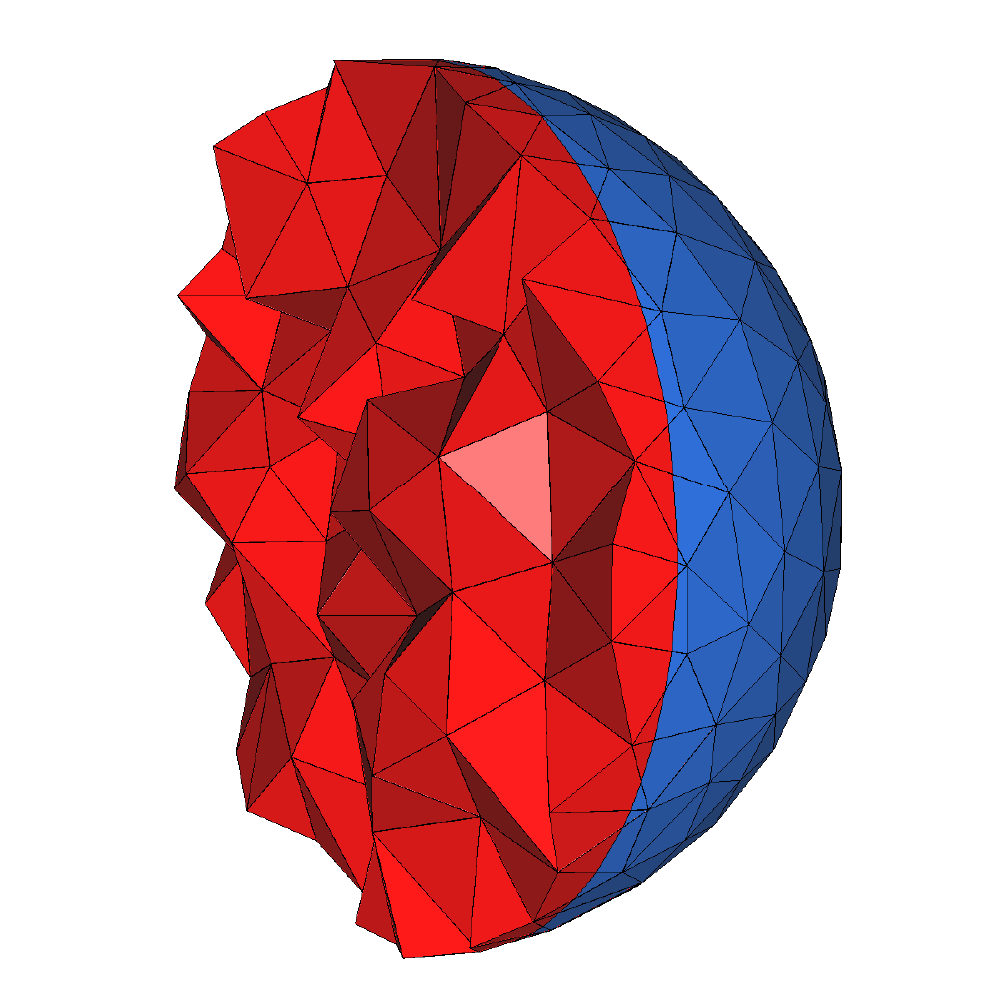}&
\includegraphics[width=0.45\textwidth]{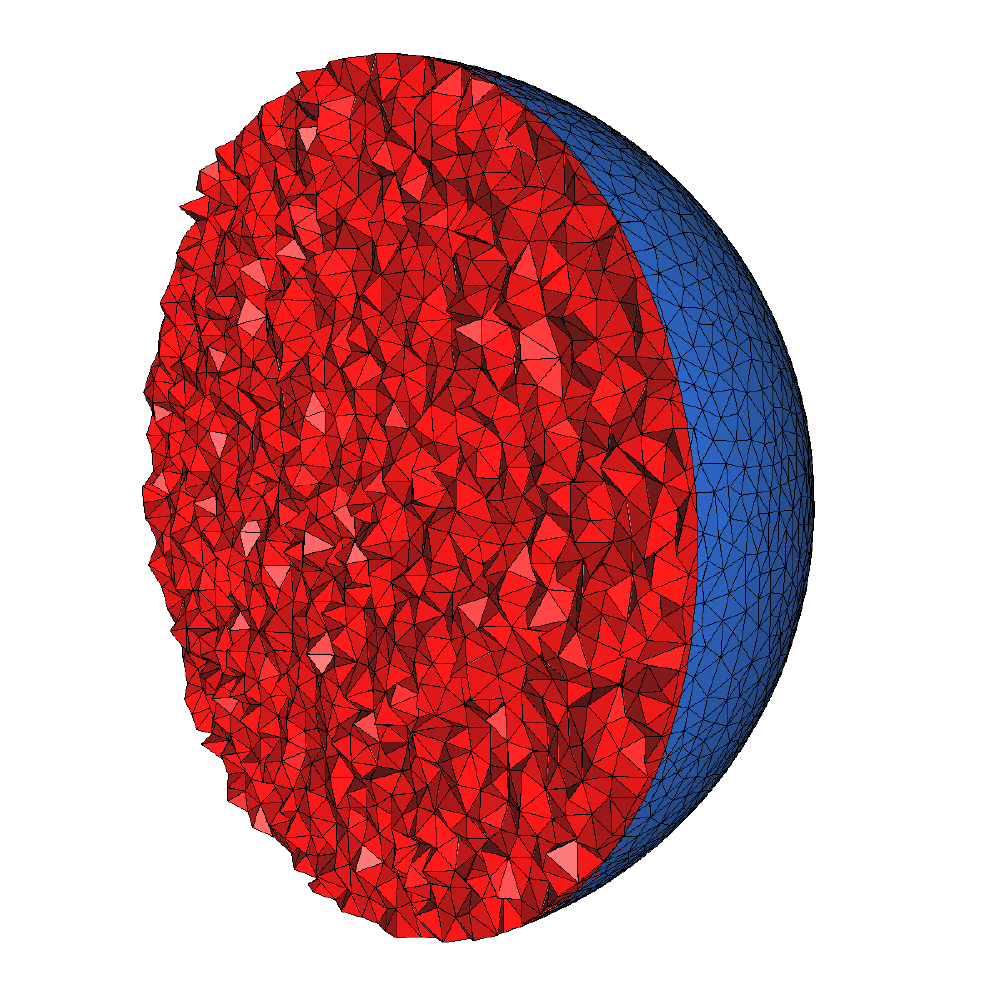}\\
(a)&(b)
\end{tabular}
\end{center}
\vspace{-4mm}
\caption{Discretization of unit ball using tetrahedron mesh. (a) a mesh with 546 vertices, (b) a mesh with 25606 vertices. }
\label{fig:unit_ball}
\end{figure}

\begin{figure}[!t]
\begin{center}
\begin{tabular}{cc}
\includegraphics[width=0.48\textwidth]{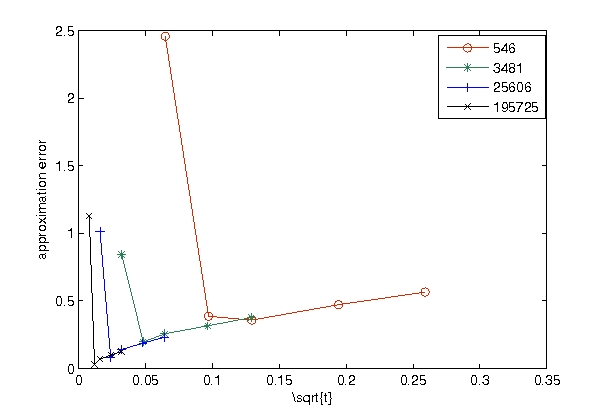}
&
\includegraphics[width=0.48\textwidth]{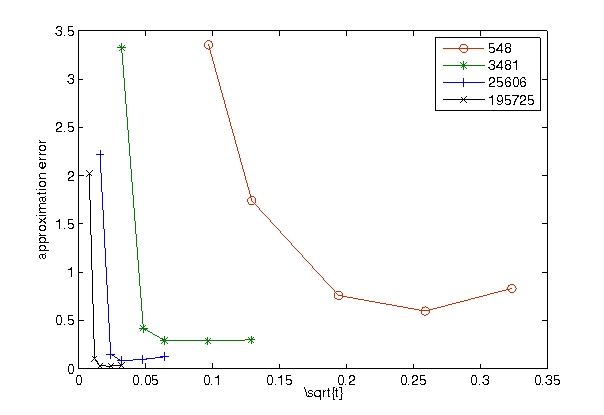}
\\
(a) & (b)
\end{tabular}
\end{center}
\vspace{-4mm}
\caption{Approximation error vs. parameter $t$ on unit ball: (a) the Neumann boundary; (b) the Dirichlet boundary. }
\label{fig:unit_ball_param_t}
\end{figure}

\vspace{0.1in}
\noindent{\bf Choice of Parameters:~}
What is good choice of $\beta$ is clear from the previous experiments on unit disk.  In fact, we observe the same 
effect of the parameter $\beta$ over the domain of unit ball, and thus we fix $\beta = 10^{-4}$ for the remaining experiments. 
Similar to the disk case, we set the boundary condition (Neumann and Dirichlet) as that of the function 
$u_{gt} = \cos 2 \pi r$ with $r = \sqrt{x^2 + y^2+ z^2}$ and see how accurate our algorithm can recover this 
function.

Figure~\ref{fig:unit_ball_param_t} shows the plot of the approximation errors $\|u-u_{gt}\| / \|u_{gt}\|$
as a function of the parameter $\sqrt{t}$. The approximating solution $u$ is computed
by Algorithm $2$ for the Neumann boundary and by Algorithm $4$ for the Dirichlet boundary.
Given a sampling $P$ on $\M$, let $\delta_i$ be the average distance 
from $p_i\in P$ to its $15$ nearest neighbors in $P$ and $\delta$ is the average of $\delta_i$ 
over all points $p_i\in P$. From the above plot, we observe that the best parameter $\sqrt t$ 
is $0.375\delta$ for Neumann boundary and $0.75\delta$ for Dirichlet boundary. Similar to the disk case, 
such optimal choice of $t$ leads to much better empirical results than what is 
predicted in~\cite{SS14}. 

\vspace{0.1in}
\noindent {\bf Convergence for the Poisson Equation:~}
We fix $\sqrt{t} = 0.375\delta$ for the Neumann boundary, and $\sqrt{t}=0.75\delta$
and $\beta = 10^{-4}$ for the Dirichlet boundary. We show the convergence of PIM and also compare it to 
FEM. In FEM, we again use linear elements. Table~\ref{tbl:unit_disk_le} shows the
approximation errors for recovering the function  $\cos 2 \pi r$.
As we can see, FEM has the quadratic convergence
rate for both Neumann boundary and Dirichlet boundary. 
PIM converges in the linear order of $h$ for Neumann boundary  and 
in the order of $h^{3/2}$ for Dirichlet boundary, where $h$ is referred to mesh size. 
This convergence rate is much faster than the order $h^{1/5}$ given by the analysis in~\cite{SS14}. 

\begin{table}[!h]
\begin{center}
\begin{tabular}{| c| c | c | c | c |}
\hline
$|V|$	& 546  & 3481 & 25606 & 195725 \\ 
\hline
\multicolumn{5}{|c|}{Neumann Boundary}\\
\hline
FEM     & 0.2245 & 0.0572 & 0.0143 & 0.0036\\
\hline
PIM       & 0.3864 & 0.1978  & 0.0845 & 0.0293\\
\hline
\multicolumn{5}{|c|}{Dirichlet Boundary}\\
\hline
FEM     & 0.3896 & 0.0878 & 0.0201 & 0.0049\\
\hline
PIM       & 0.7572 & 0.2881 & 0.0952 & 0.0256\\
\hline
\end{tabular}
\end{center}
\vspace*{-0.15in}
\caption{Convergence for recovering the function $\cos 2\pi  r$. 
The solution is computed using Algorithm 2 for Neumann 
boundary and Algorithm 4 for Dirichlet boundary.
\label{tbl:unit_ball_le}}
\end{table}

\begin{figure}[!h]
\begin{center}
\begin{tabular}{cc}
\includegraphics[width=0.45\textwidth]{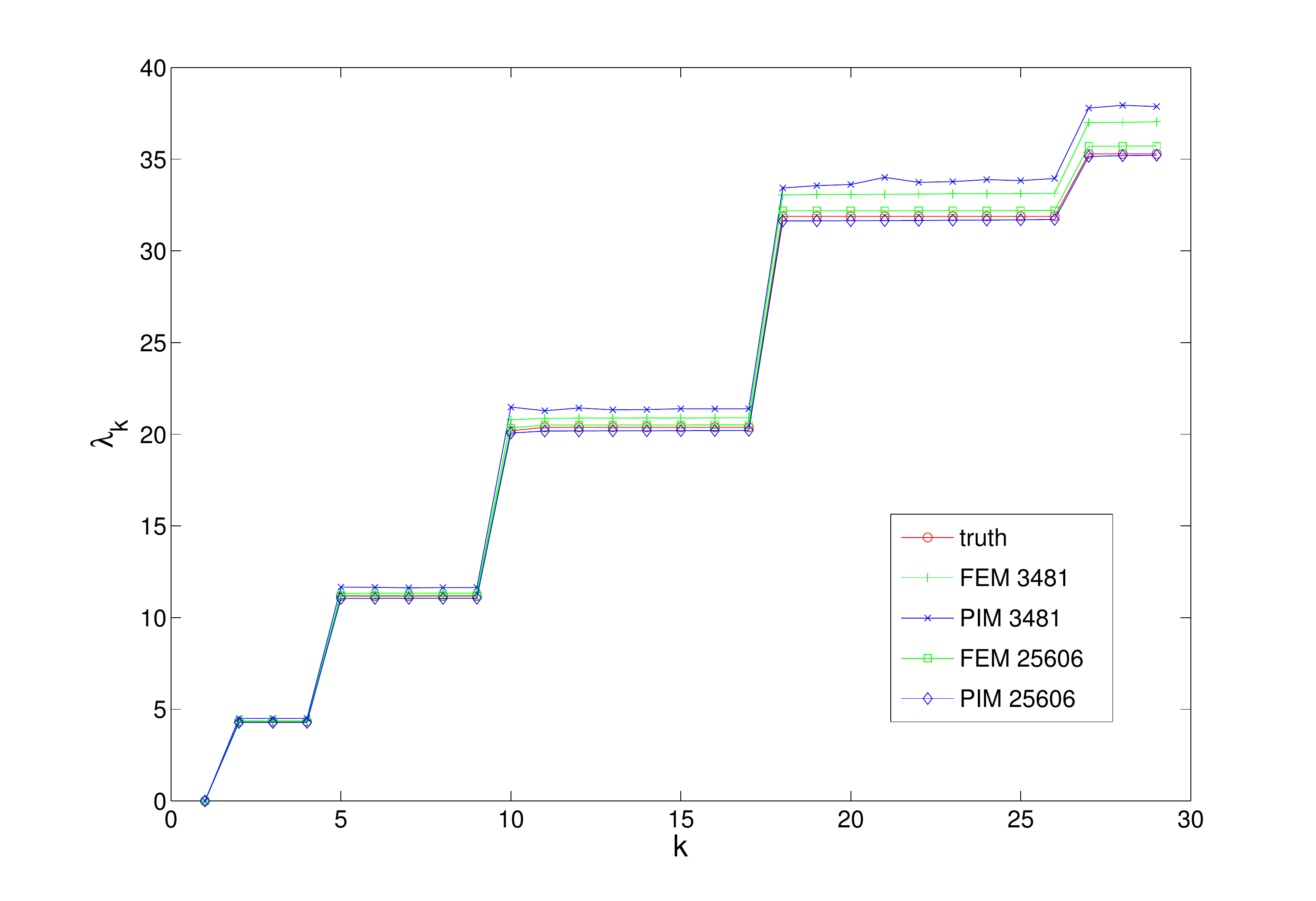}&
\includegraphics[width=0.45\textwidth]{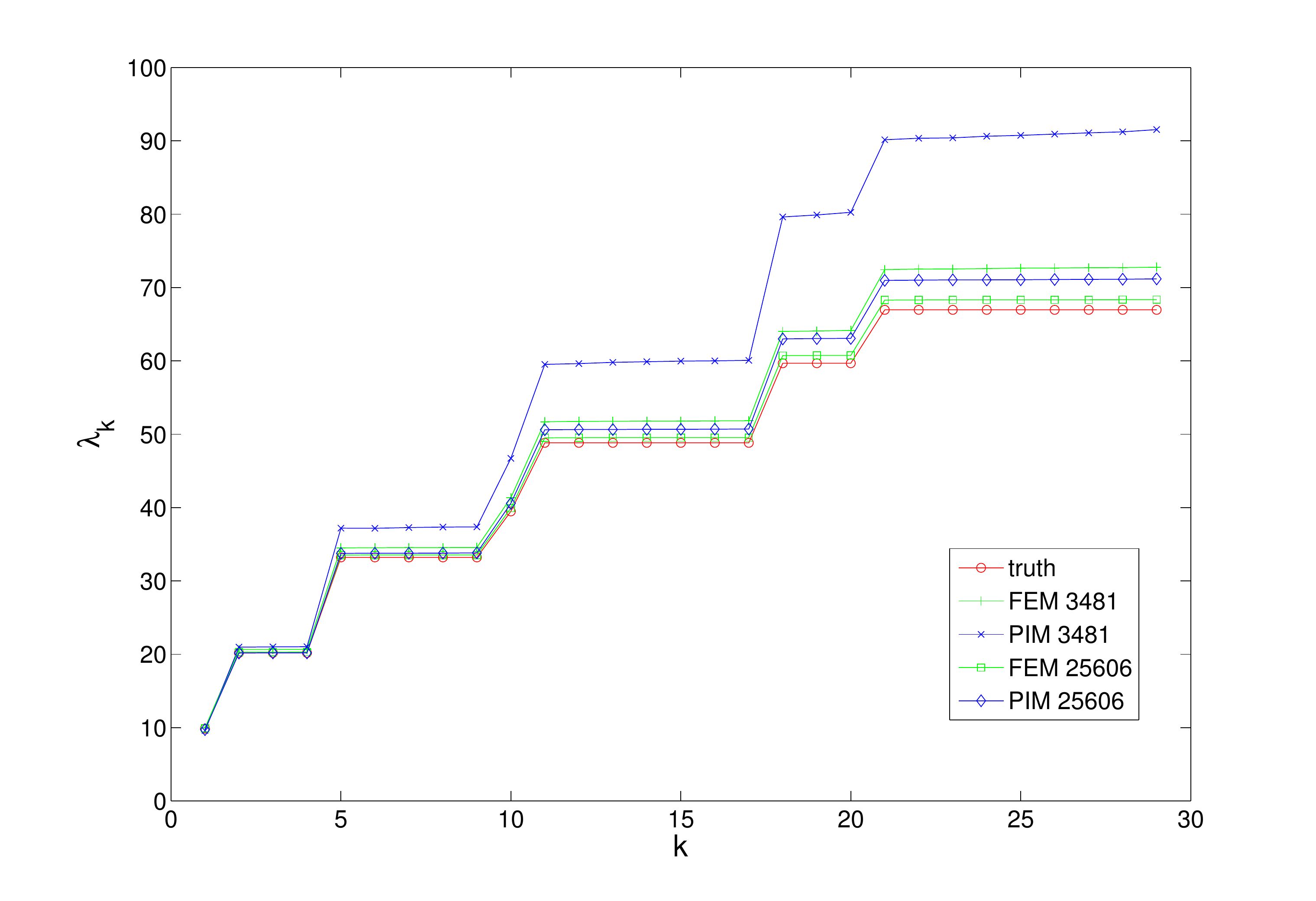}\\
(a)&(b)
\end{tabular}
\end{center}
\vspace*{-0.15in}
\caption{The eigenvalues of unit ball estimated by FEM and PIM over the meshes (samplings) with
3481 points and 25606 points. (a) Neumann eigenvalues; (b) Dirichlet eigenvalues}
\label{fig:unit_ball_eig}
\end{figure}

\begin{figure}[!h]
\begin{center}
\begin{tabular}{cc}
\includegraphics[width=0.45\textwidth]{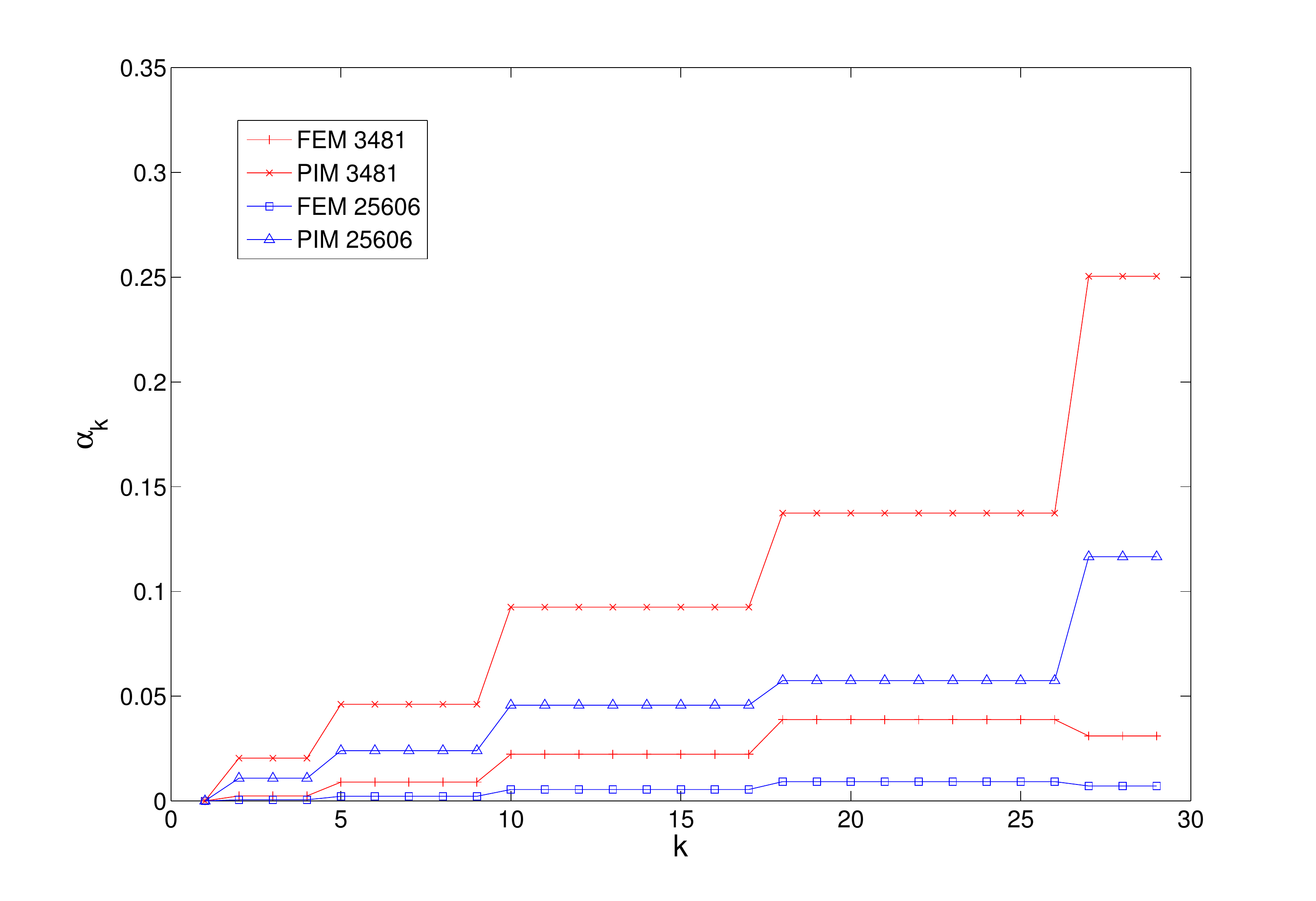}&
\includegraphics[width=0.45\textwidth]{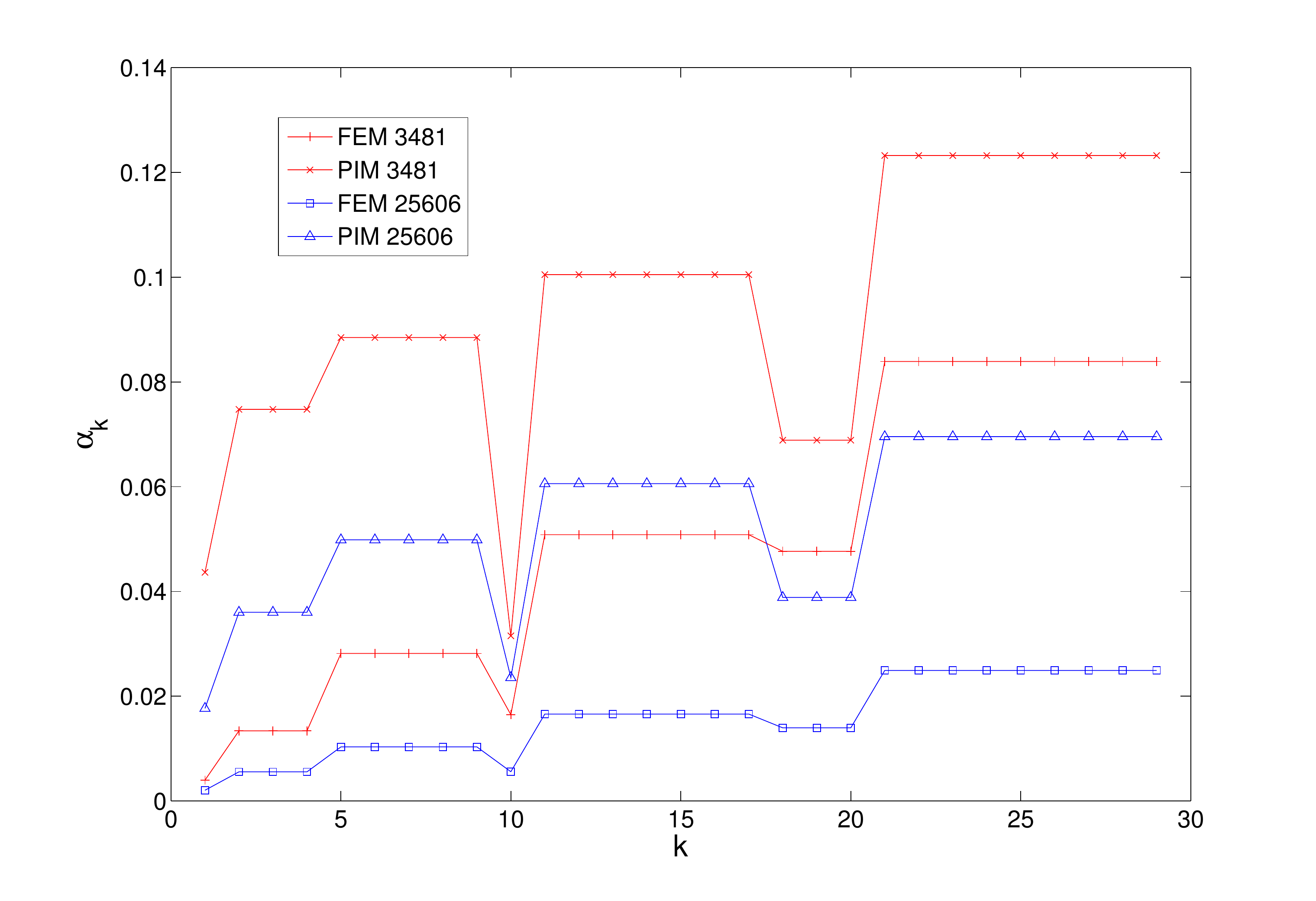}\\
(a)&(b)
\end{tabular}
\end{center}
\vspace*{-0.15in}
\caption{The approximation errors of the eigenfunctions for unit ball estimated using FEM (PIM) over the meshes (samplings) with 
3481 points and 25606 points. (a) Neumann eigenfunctions; (b) Dirichlet eigenfunctions}
\label{fig:unit_ball_eigenfunction}
\end{figure}

\vspace{0.1in}
\noindent{\bf Eigensystem:~} 
We compute the eigensystem of Laplacian using Algorithm 3 EigenNeumann for the problem~\eqref{eigen-neumann}
and Algorithm 5 EigenDirichlet for the problem~\eqref{eigen-dirichlet}. 
We choose the parameters as before. 
Figure~\ref{fig:unit_ball_eig} shows the first $30$ eigenvalues computed using PIM (FEM) over 
the sampling (mesh) with 3481 points and 25606 points. Both methods give a good estimation for the eigenvalues. 
Figure~\ref{fig:unit_ball_eigenfunction} shows the approximation error of the first 30 eigenfunctions, 
where the approximation error is computed as before, i.e., the angle between the eigenspaces of ground 
truth and the eigenspaces estimated by PIM or FEM (see Equation~\eqref{eqn:erroreigenfunction}). 

\begin{figure}[!t]
\begin{center}
\begin{tabular}{cc}
\includegraphics[width=0.45\textwidth]{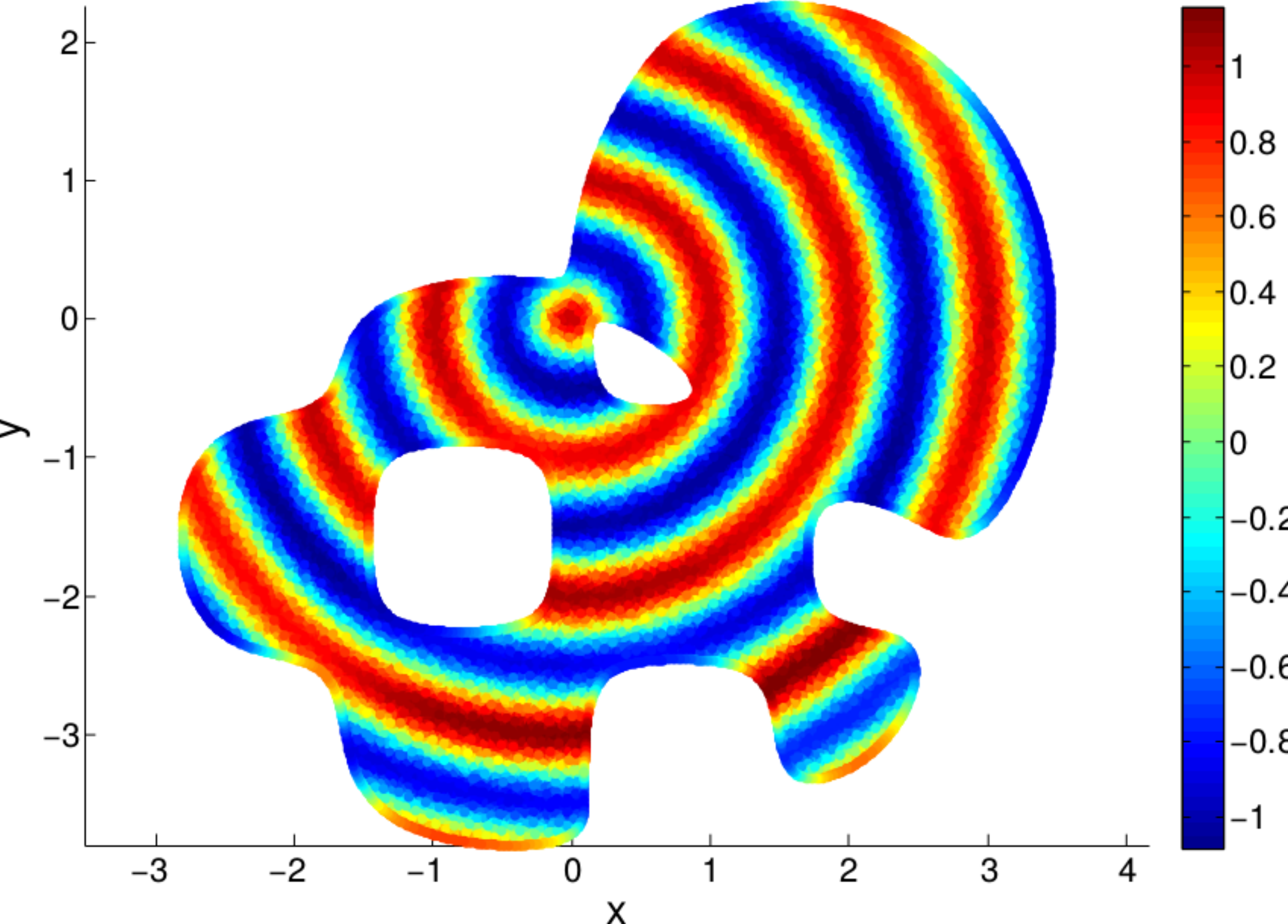}&
\includegraphics[width=0.45\textwidth]{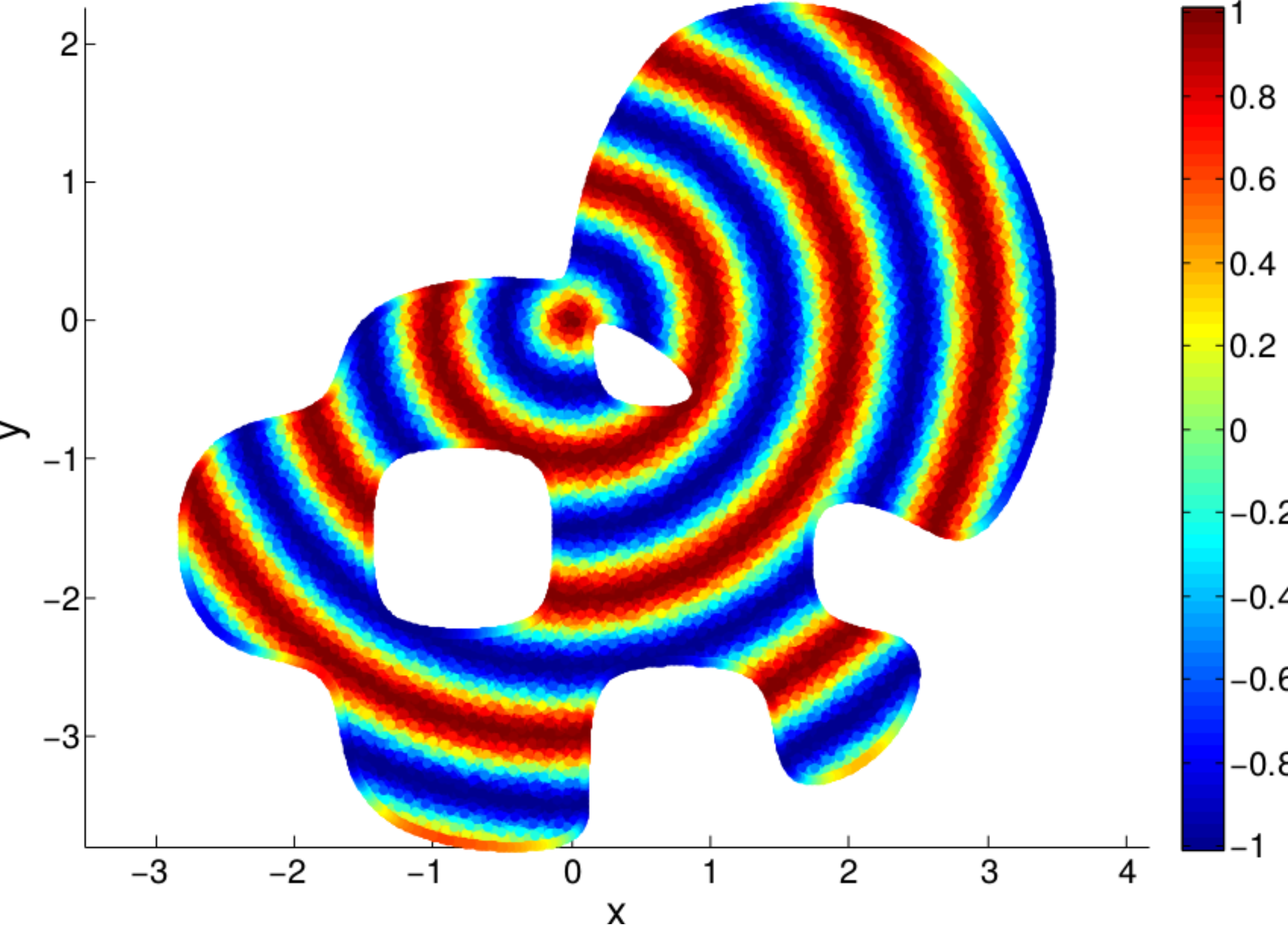}\\
(a) & (b)
\end{tabular}
\end{center}
\vspace*{-0.15in}
\caption{The function $u_{gt} = \cos 2 \pi r$ recovered by PIM for Lake from a point cloud with 64131 points: 
(a) Neumann boundary; (b) Dirichlet boundary}
\label{fig:lake}
\end{figure}

\subsection{An irregular Domain on Plane}
We also test PIM over an irregular domain: a planar domain with two holes, 
which we call Lake as it looks like a lake with two islands. 
Again, we set the boundary condition (both Neumann and Dirichlet) as that of the function 
$u_{gt} = \cos 2 \pi r$ with $r = \sqrt{x^2 + y^2}$ and see how accurate our algorithm
can recover this function.
Based on the previous experiments on unit disk, we fix the parameter $\sqrt{t} = 0.75 \delta$ and $\beta = 10^{-4}$. 
Figure~\ref{fig:lake} shows the functions recovered by PIM. 
The approximation errors of the solutions are listed in Table~\ref{tbl:lake_le}.
As we can see, PIM converges in the linear order for the problem with the Neumann boundary and 
in the order of $h^{3/2}$ for the problem with the Dirichlet boundary. In addition, the solution
for the Neumann boundary has bigger approximation errors than that for the Dirichlet boundary.

\begin{table}[!h]
\begin{center}
\begin{tabular}{| c| c | c | c | c |}
\hline
$|V|$		& 942 & 3745 & 16034 & 64131 \\ 
\hline
Neumann     	& 4.4595 & 1.8360 & 0.5842 & 0.3029\\
\hline
Dirichlet       & 0.3030 & 0.0907 & 0.0247 & 0.0081\\
\hline
\end{tabular}
\end{center}
\vspace*{-0.15in}
\caption{Convergence of PIM in solving the Poisson Equations over Lake\label{tbl:lake_le}}
\end{table}

The ground truth of the eigenvalues and the eigenfunctions of Lake can not be expressed
in an explicit way. Thus we compare the results of PIM with that of FEM. 
Figure~\ref{fig:lake_eigenvalue} shows the first 30 eigenvalues of Lake estimated by PIM and FEM.
Both methods give a consistent estimation of the eigenvalues of Lake. 
Figure~\ref{fig:lake_eigenfunction} shows the 10th eigenfunction of Lake estimated by PIM from a point cloud
with 64131 points.  
Table~\ref{tbl:lake_eigenfunction} shows the relative $L^2$ errors $\|u_{PIM}-u_{FEM}\|/\|u_{FEM}\|$ 
of the 10th  eigenfunction estimated from different point clouds or meshes. As we can see, 
the relative errors decrease as the sample points increases for both boundary conditions. 
\begin{figure}[!t]
\begin{center}
\begin{tabular}{cc}
\includegraphics[width=0.45\textwidth]{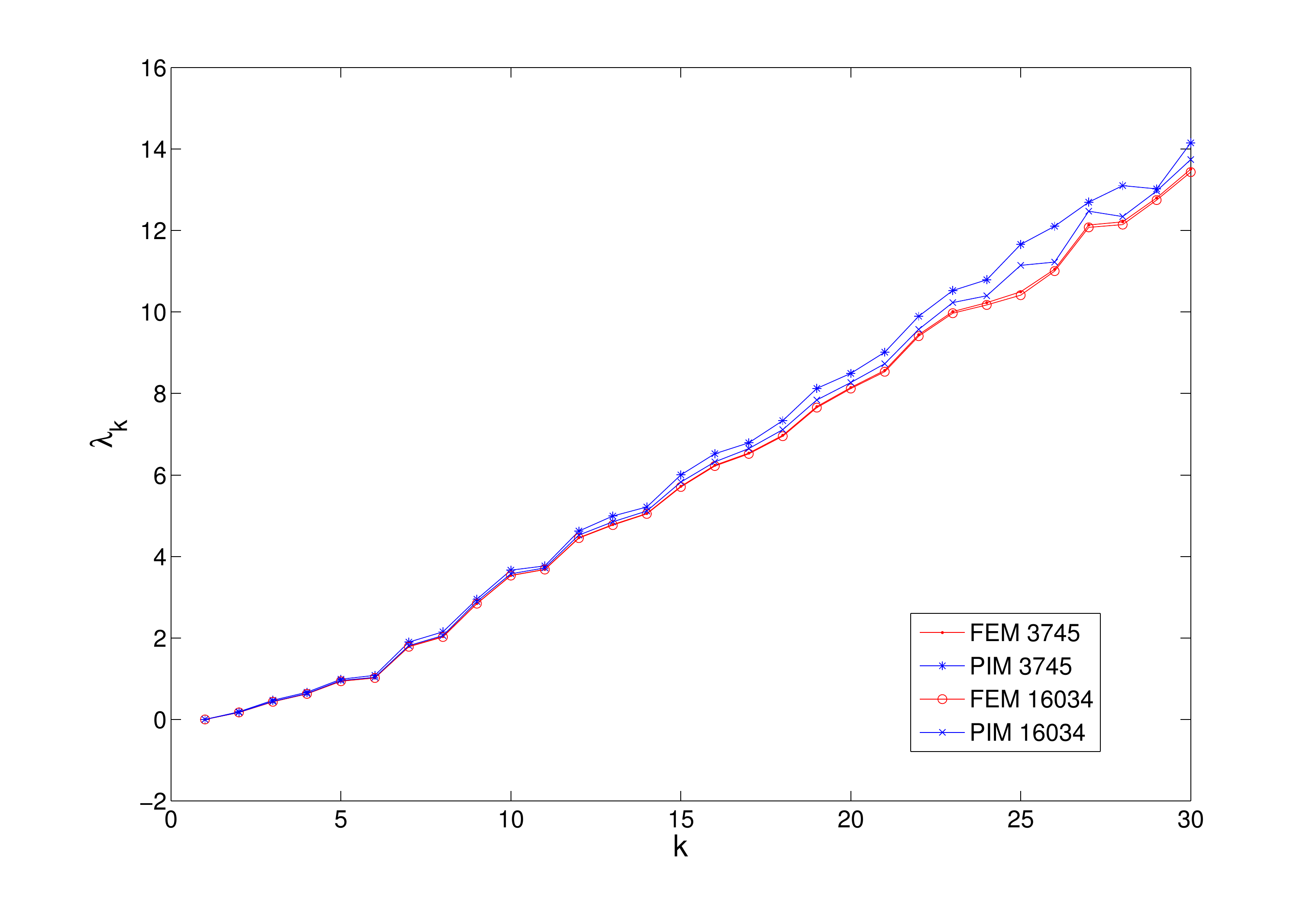}
&
\includegraphics[width=0.45\textwidth]{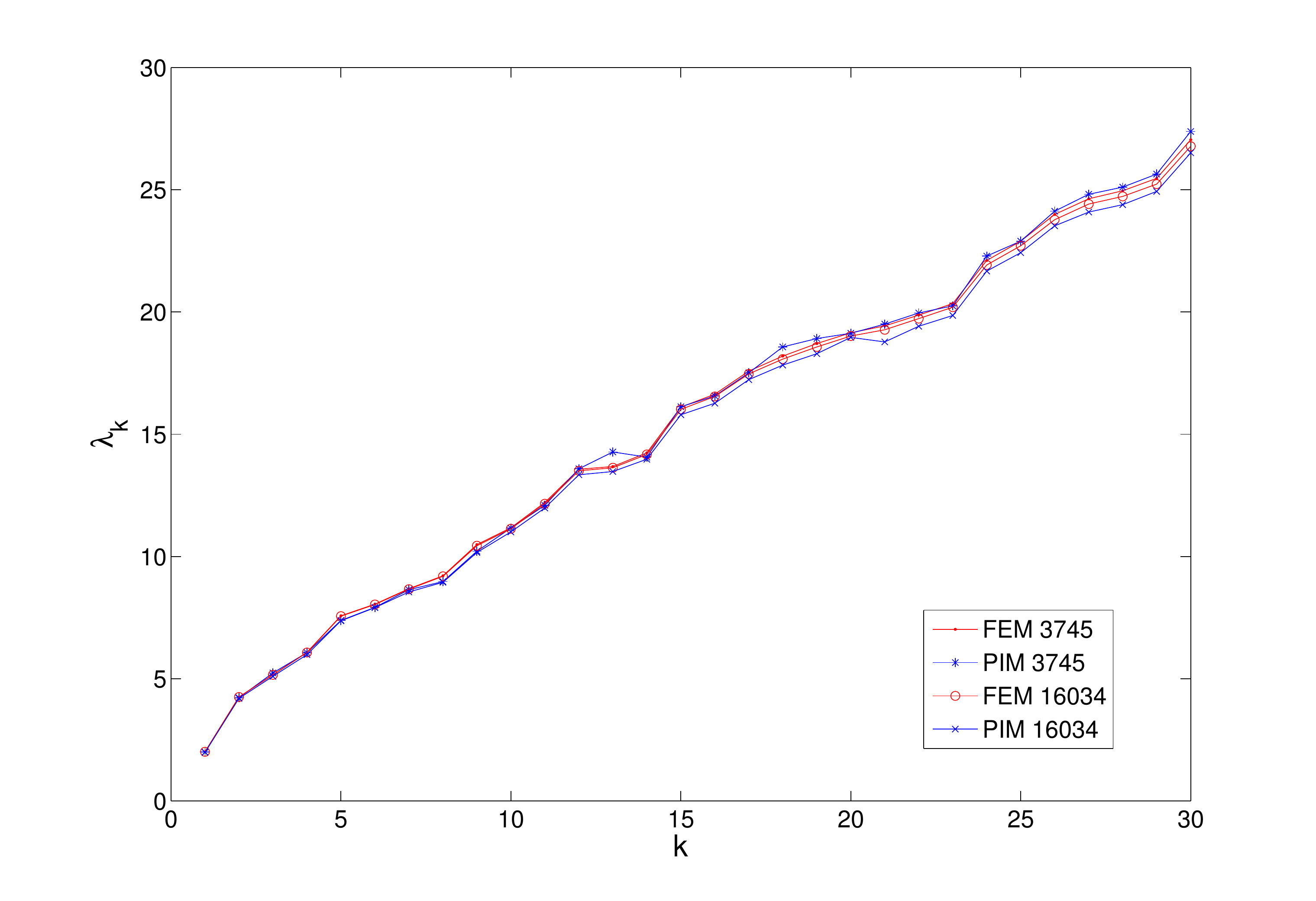}
\\
(a) & (b)
\end{tabular}
\end{center}
\vspace*{-0.15in}
\caption{The eigenvalues of Lake estimated by PIM and FEM with 3745 points and 16034 points: (a) Neumann problem; (b) Dirichlet problem}
\label{fig:lake_eigenvalue}
\end{figure}

\begin{figure}[!t]
\begin{center}
\begin{tabular}{cc}
\includegraphics[width=0.45\textwidth]{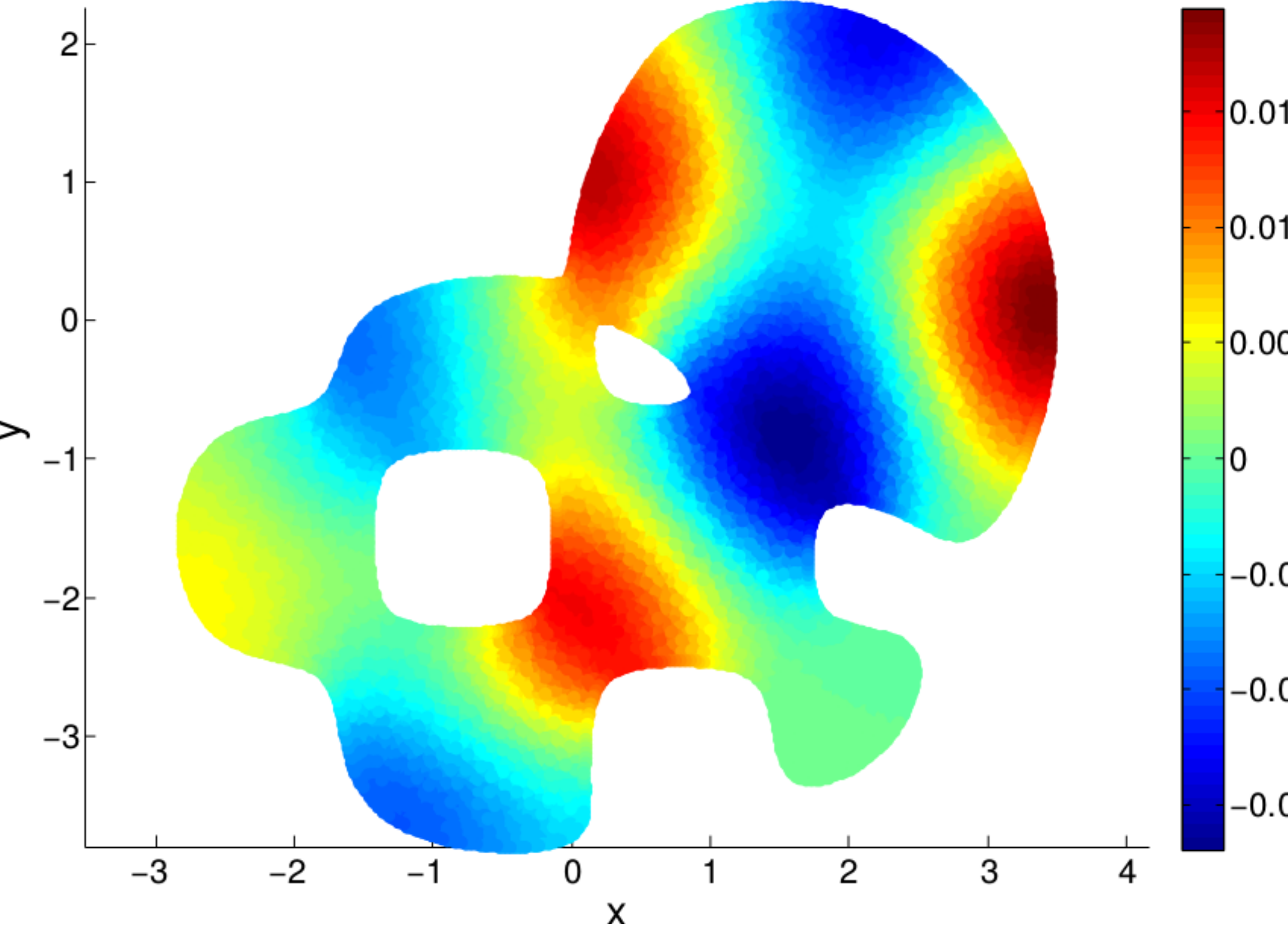}
&
\includegraphics[width=0.45\textwidth]{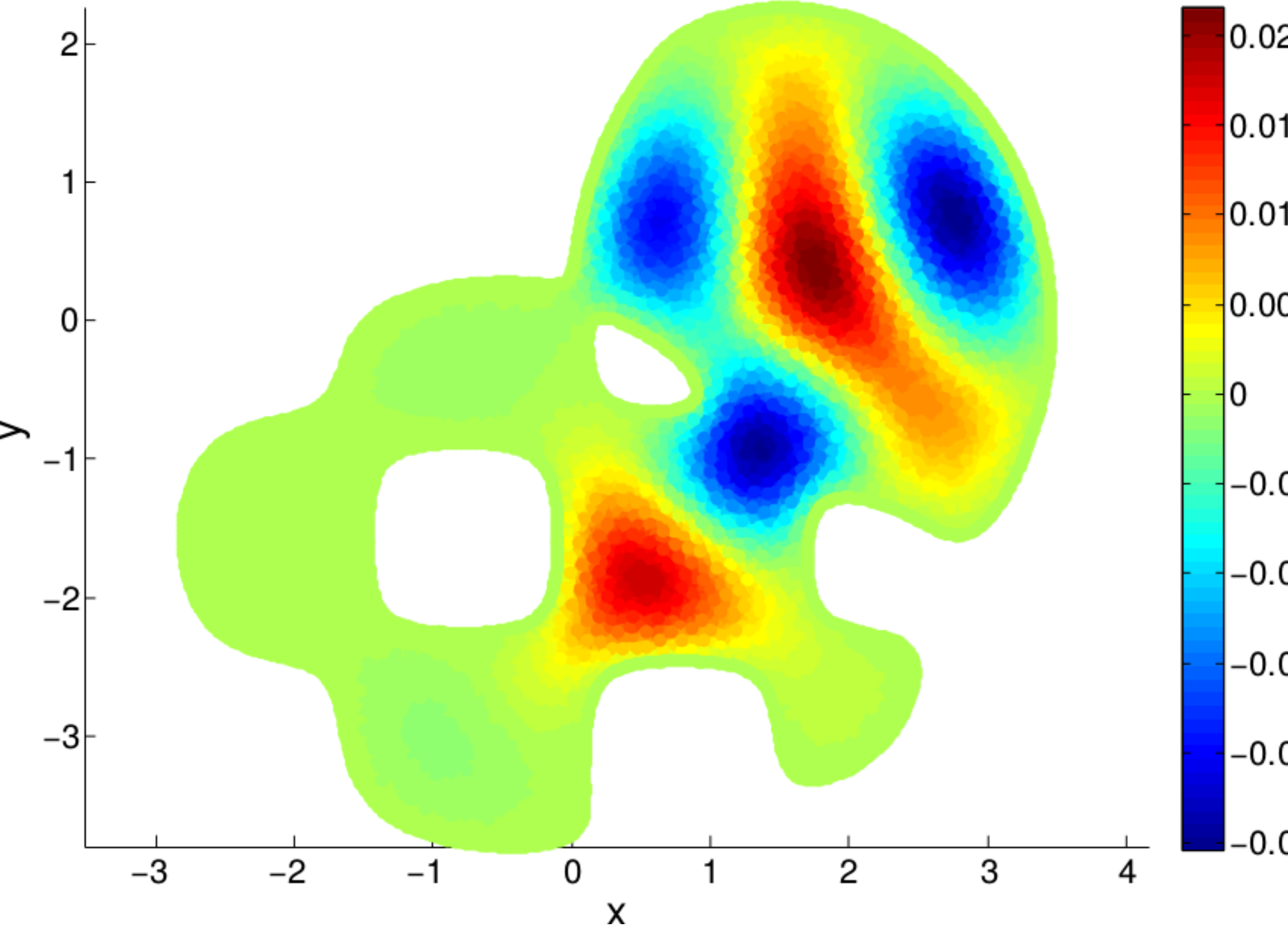}
\\
(a) & (b)
\end{tabular}
\end{center}
\vspace*{-0.15in}
\caption{The 10th eigenfunction of Lake estimated by PIM from a point cloud with 64131 points: (a) Neumann problem; (b) Dirichlet problem}
\label{fig:lake_eigenfunction}
\end{figure}

\begin{table}[!t]
\begin{center}
\begin{tabular}{| c | c | c | c | c |}
\hline
Number of points  & 942 & 3745 & 16034 & 64131\\
\hline
Neumann Boundary  & 0.0336 & 0.0124 & 0.0046 & 0.0022\\
\hline
Dirichlet Boundary& 0.0893 & 0.0329 & 0.0088 & 0.0033\\
\hline
\end{tabular}
\end{center}
\caption{The relative $L^2$ errors $\|u_{PIM}-u_{FEM}\|/\|u_{FEM}\|$ 
of the 10th eigenfunction of Lake. \label{tbl:lake_eigenfunction}}
\end{table}

\subsection{General Submanifolds}
In this subsection, we apply PIM to solve the Poisson equations 
on a few examples of general submanifolds. In the following experiments, we fix 
$\sqrt{t}=0.75\delta$ and $\beta = 10^{-4}$.  

\begin{figure}[!t]
\begin{center}
\begin{tabular}{cc}
\includegraphics[width=0.45\textwidth]{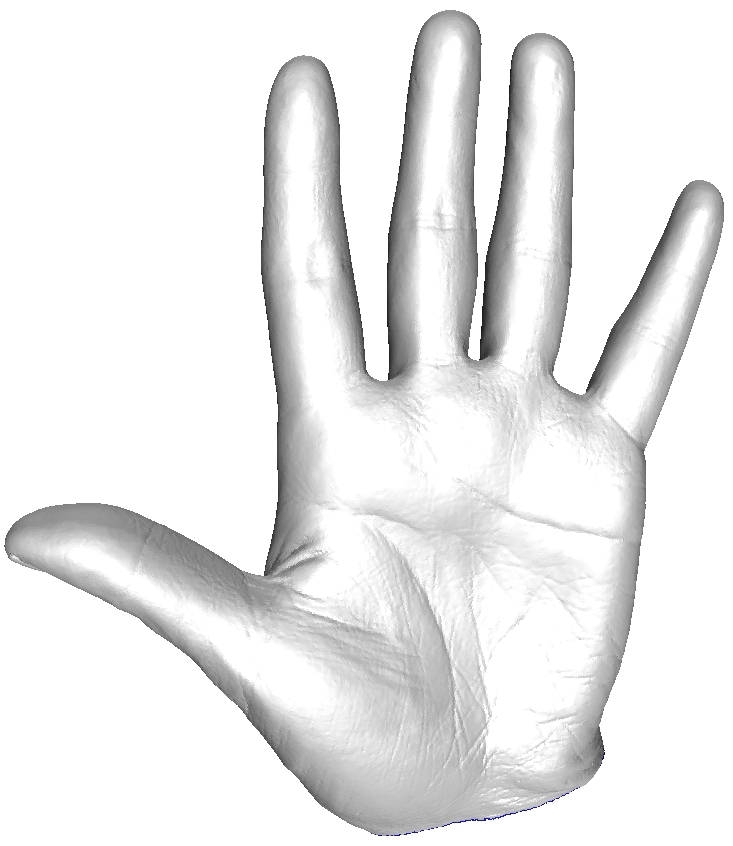} &
\includegraphics[width=0.45\textwidth]{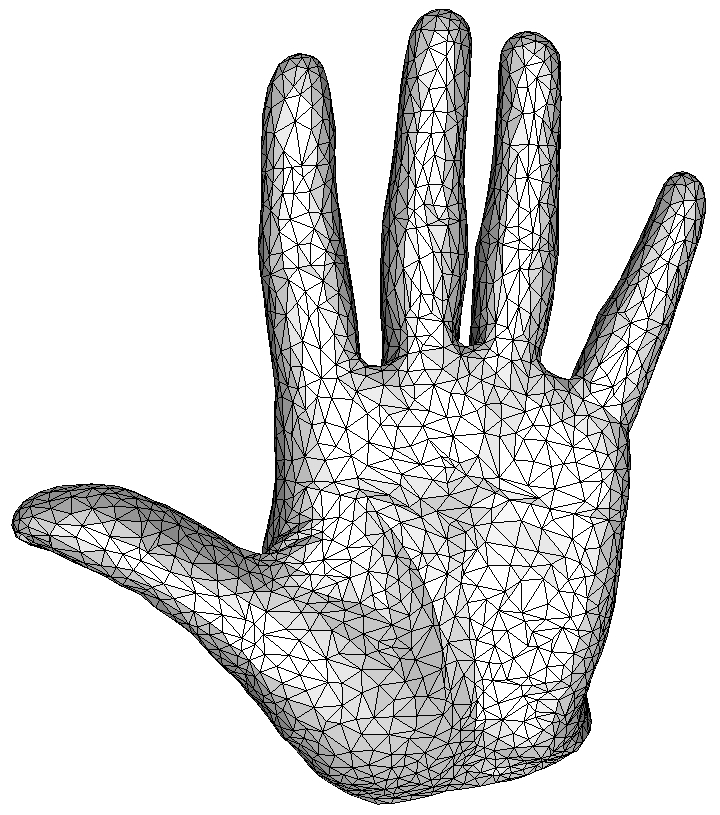}\\
(a) & (b)
\end{tabular}
\end{center}
\vspace{-4mm}
\caption{Lefthand (a) the original model; (b) the mesh with 3147 points}
\label{fig:lefthand_model}
\end{figure}

The first example is a model (Lefthand) of the left hand of a human obtained by 3D scanning. 
The original model is a triangle mesh with $193467$ vertices, as shown in Figure~\ref{fig:lefthand_model}(a). 
We use Meshlib~\cite{meshlab} to simplify the mesh to obtain the triangle meshes with $50205$, 
$12561$ and $3147$ vertices, over which FEM is applied to solve the 
Poisson equation. Figure~\ref{fig:lefthand_model}(b) shows the mesh with $3147$ vertices. 
For PIM, the vertices of the meshes are taken 
as the input point sets $P$, and those on the boundary are taken as the input point sets $S$.
As there is no analytic solution of the Poisson equation for a general manifold, we compare the 
solutions from FEM and PIM to each other, and show that they are consistent to each other. 
We solve over the model Lefthand the following Neumann problem
\begin{eqnarray}
  \left\{\begin{array}{rl}
      -\Delta u(\bx) = |\bx|^2 , &\bx\in \mathcal{M},\\
      \frac{\p u}{\p \bn}(\bx)=1,&\bx\in \p \mathcal{M},
\end{array}\right.
\label{hand_neumann}
\end{eqnarray}
and the following Dirichlet problem
\begin{eqnarray}
  \left\{\begin{array}{rl}
      -\Delta u(\bx) = |\bx|^2 , &\bx\in \mathcal{M},\\
      u(\bx)=1,&\bx\in \p \mathcal{M}.
\end{array}\right.
\label{hand_dirichlet}
\end{eqnarray}

Figure~\ref{fig:lefthand_poisson} shows the solutions computed by point integral method over
the point set with $193467$ points. 
Table~\ref{tbl:lefthand_possion} shows the approximation errors computed as $\|u_{PIM}-u_{FEM}\|/\|u_{FEM}\|$, 
where $u_{PIM}$ and $u_{FEM}$ are the solutions computed by PIM
and FEM respectively. The error decreases as the number of points increases, 
which shows both solutions from PIM and FEM are consistent. 

\begin{figure}[!t]
\begin{center}
\begin{tabular}{cc}
\includegraphics[width=0.4\textwidth]{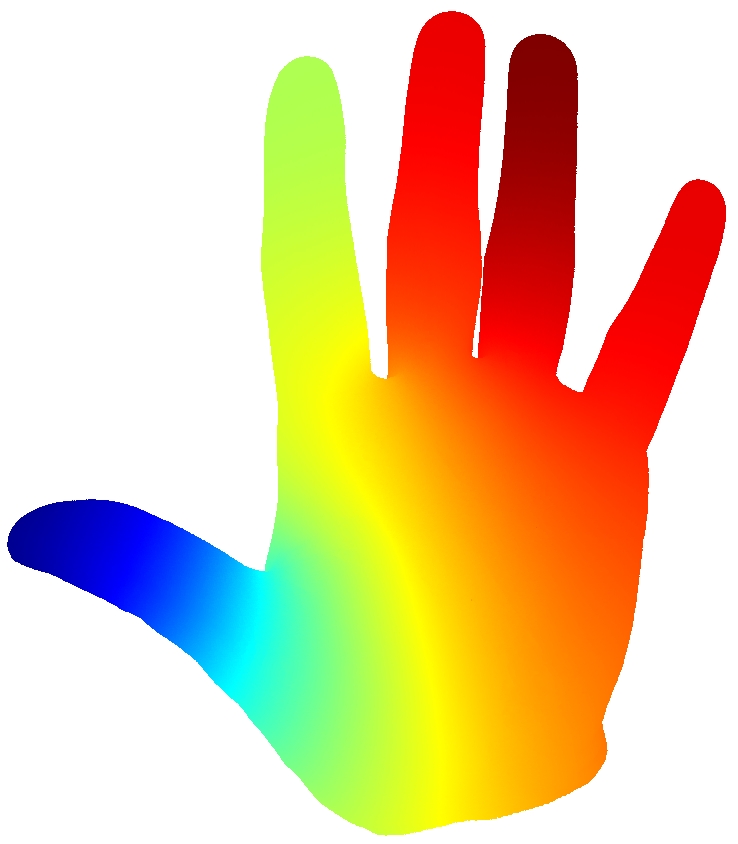} &
\includegraphics[width=0.4\textwidth]{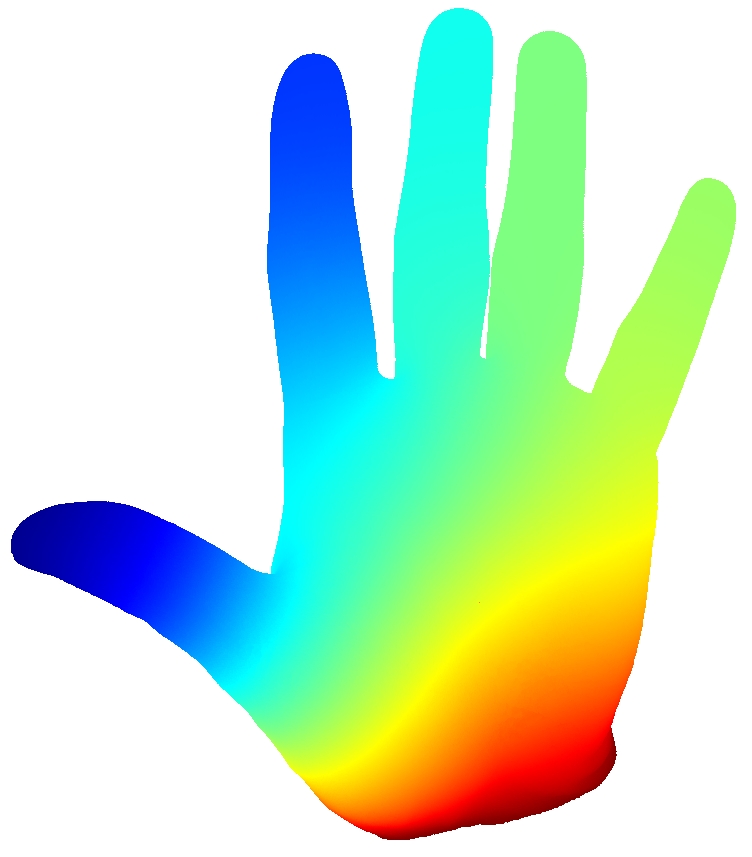}\\
(a) & (b)
\end{tabular}
\end{center}
\vspace{-4mm}
\caption{(a) Neumann problem~\eqref{hand_neumann}; (b) Dirichlet problem~\eqref{hand_dirichlet}}
\label{fig:lefthand_poisson}
\end{figure}

\begin{table}[!h]
\begin{center}
\begin{tabular}{| c | c | c | c | c|}
\hline
$|V|$  & 3147   & 12561  & 50205 & 193467\\
\hline
Neumann  & 0.5944 & 0.0687 & 0.0594 & 0.0035\\
\hline
Dirichlet& 0.7109 & 0.0259 & 0.0229 & 0.0067\\
\hline
\end{tabular}
\end{center}
\caption{The approximation errors $\|u_{PIM}-u_{FEM}\|/\|u_{FEM}\|$ where
$u_{PIM}$ and $u_{FEM}$ are the solutions of the Poisson equations \eqref{hand_neumann} 
and \eqref{hand_dirichlet} computed by PIM and FEM respectively. 
\label{tbl:lefthand_possion}}
\end{table}

Figure~\ref{fig:lefthand_eigen} shows the first $30$ eigenvalues of Lefthand
using FEM over the mesh of $193467$ vertices and using PIM over the 
samplings with different number of points. As we can see, PIM can accurately estimate the 
eigenvalues of the Laplace-Beltrami operator with both the Neumann boundary and 
the Dirichlet boundary. Finally, Figure~\ref{fig:various_eigen} shows the 10th eigenfunction 
estimated by PIM over various models.  

\begin{figure}[!t]
\begin{center}
\begin{tabular}{cc}
\includegraphics[width=0.45\textwidth]{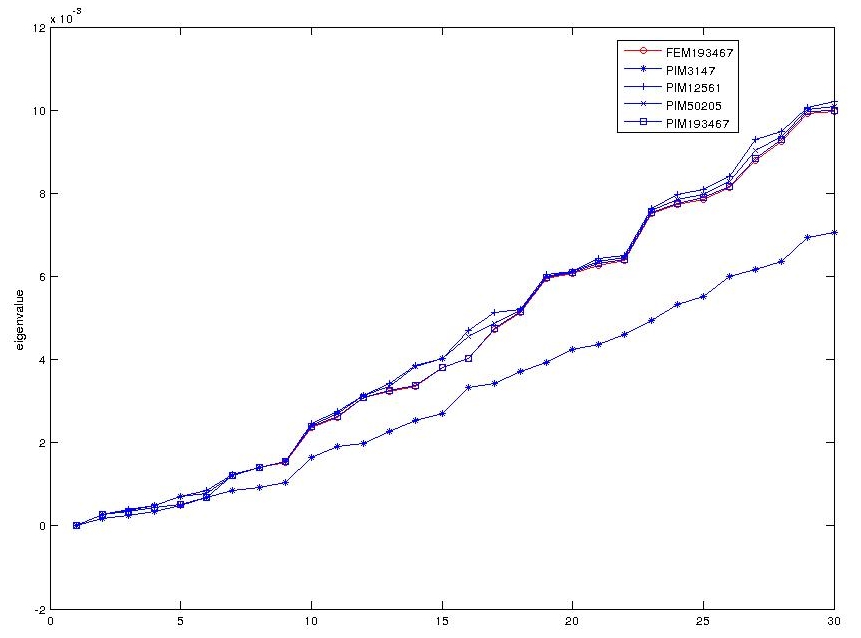} &
\includegraphics[width=0.45\textwidth]{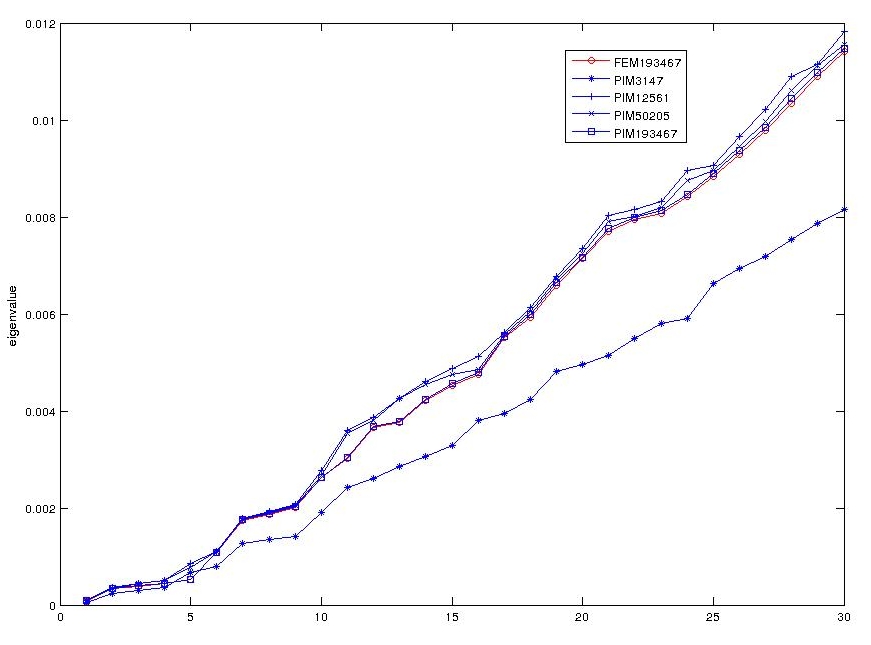}\\
(a) & (b)
\end{tabular}
\end{center}
\vspace{-4mm}
\caption{(a) Neumann problem; (b) Dirichlet problem}
\label{fig:lefthand_eigen}
\end{figure}
\begin{figure}[!t]
\begin{center}
\begin{tabular}{cccc}
\includegraphics[width=0.24\textwidth]{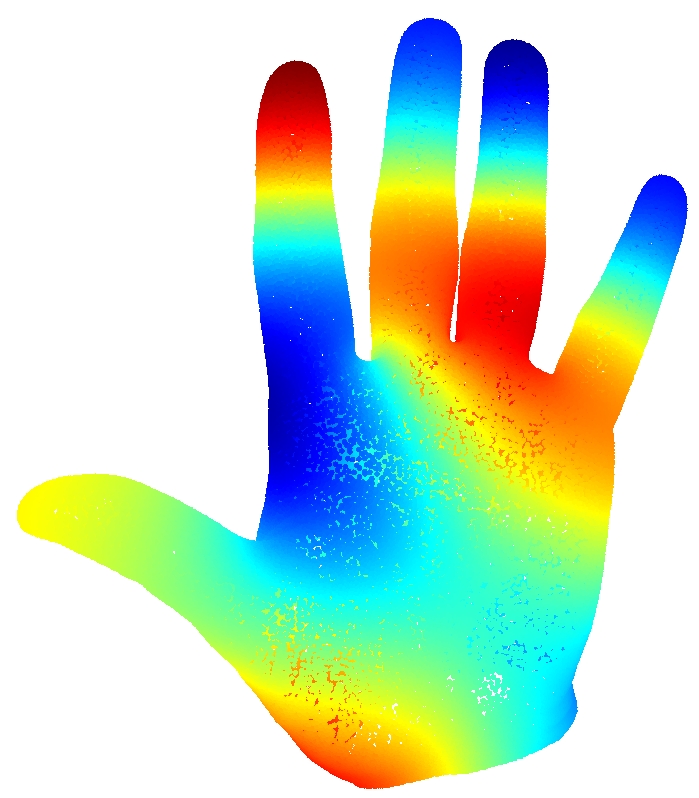} &
\includegraphics[width=0.20\textwidth]{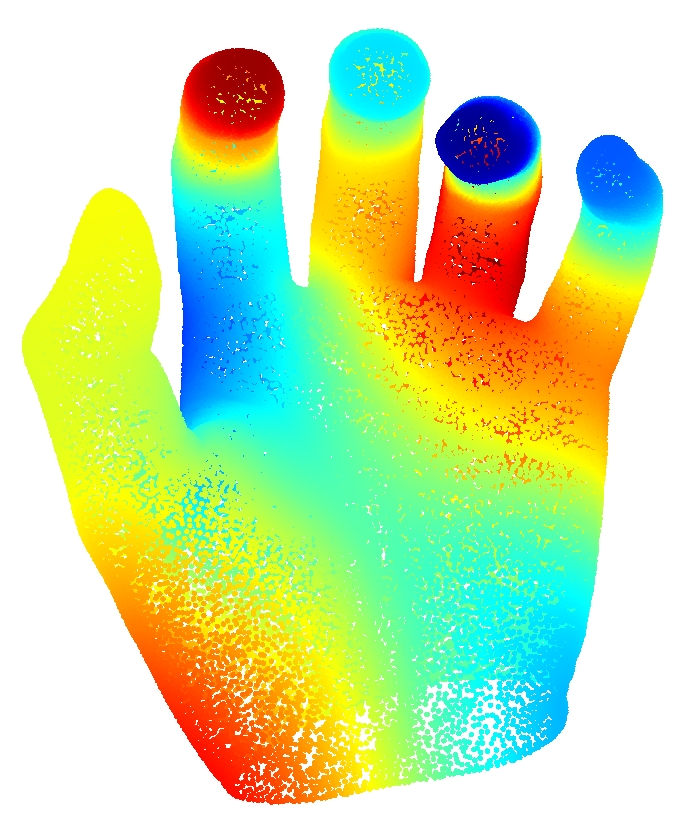} &
\includegraphics[width=0.24\textwidth]{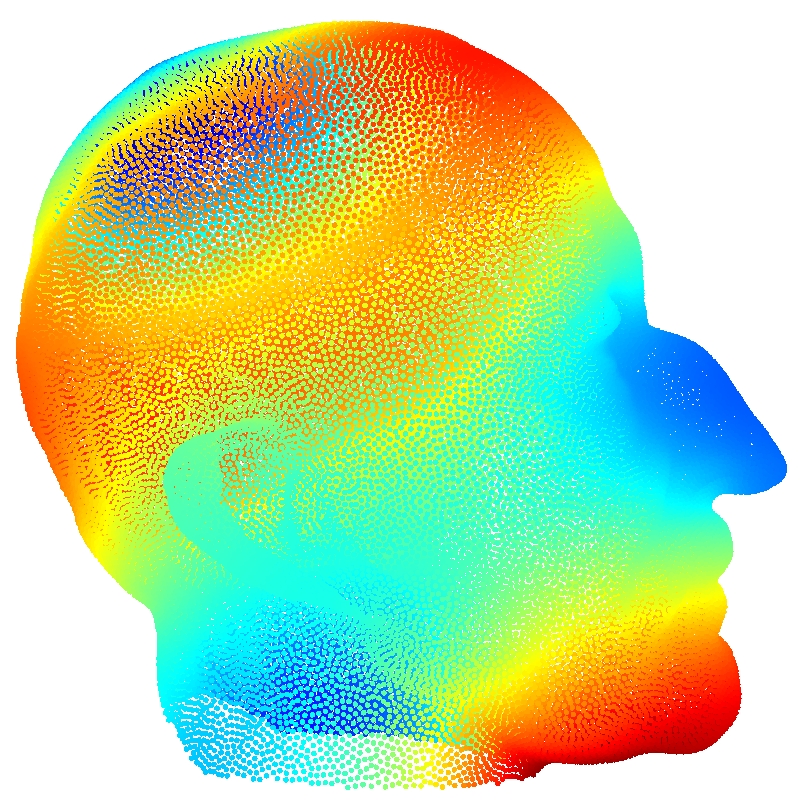} &
\includegraphics[width=0.22\textwidth]{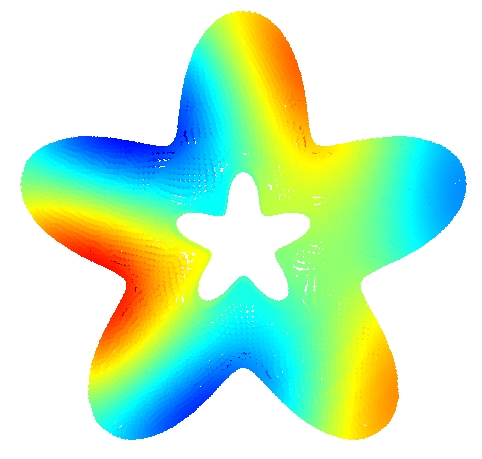} \\
\includegraphics[width=0.24\textwidth]{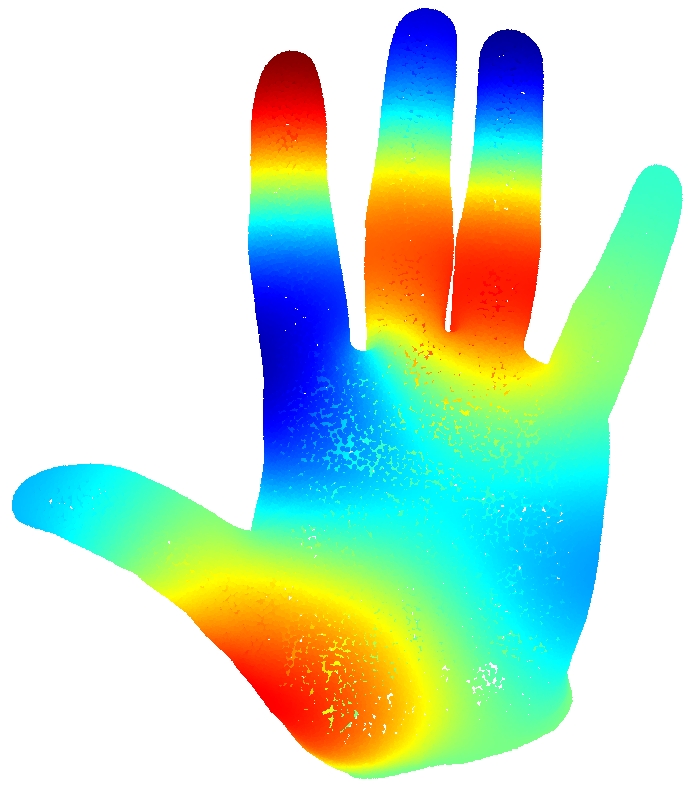} &
\includegraphics[width=0.20\textwidth]{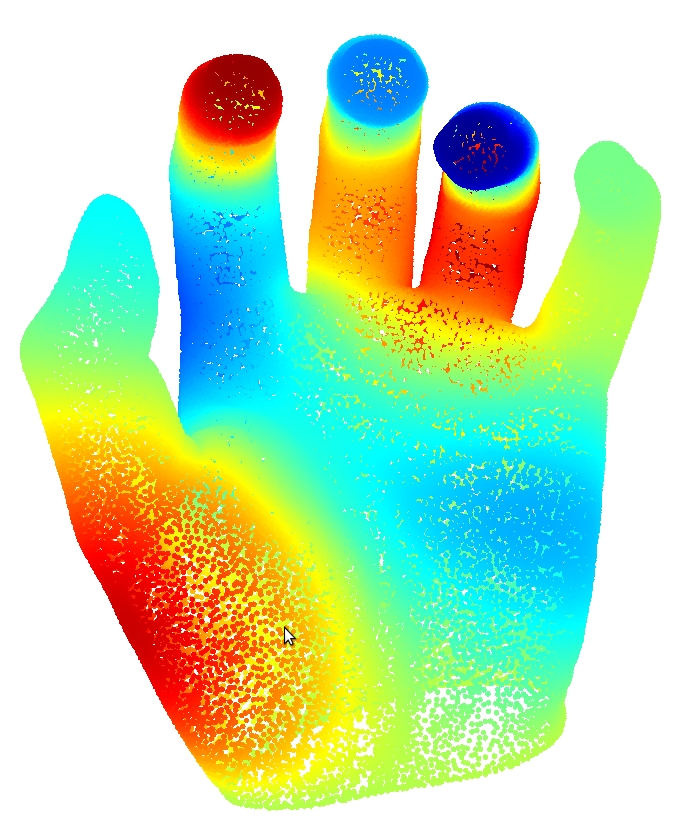} &
\includegraphics[width=0.24\textwidth]{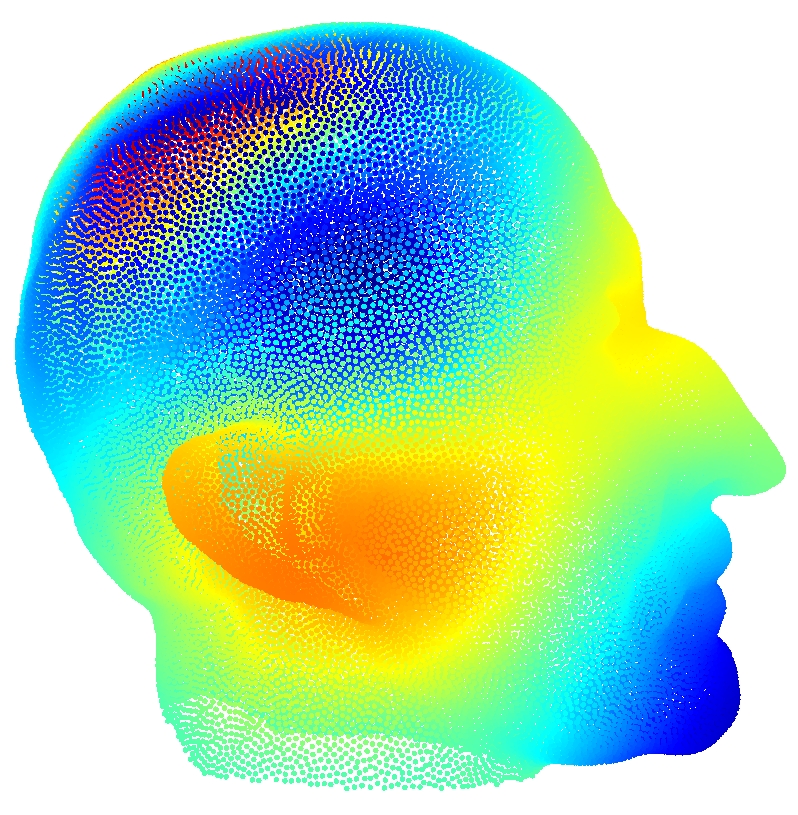} &
\includegraphics[width=0.20\textwidth]{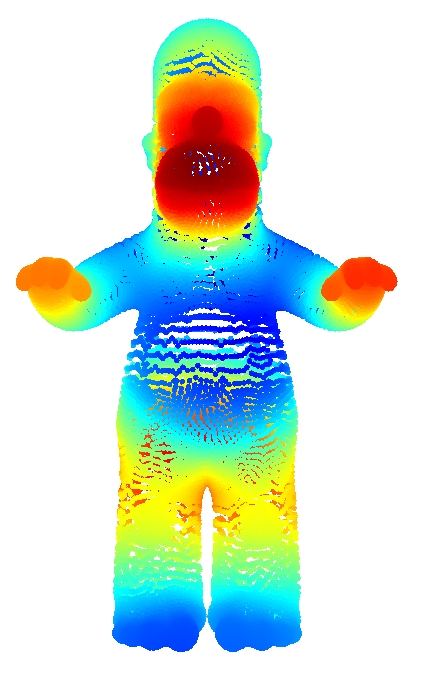} \\
\end{tabular}
\end{center}
\vspace{-4mm}
\caption{The 10th eigenfunction: Neumann boundary in the first row and Dirichlet boundary in the second row. 
Two models in the rightmost column have no boundary.}
\label{fig:various_eigen}
\end{figure}

\section{Conclusion}
\label{sec:discussion}
We have described the point integral method for solving the standard Poisson equation
on manifolds and the eigensystem of the Laplace-Beltrami operator from point clouds, 
and presented a few numerical examples, which not only demonstrate the convergence 
of PIM in solving the Poisson-type equations, but also reveal the right choices of 
the parameters $t$ and $\beta$ used in PIM. In addition, the numerical experiments 
show PIM has a faster empirical convergence rate than what is predicted by the 
analysis in~\cite{SS14}, which suggests that the analysis may be improved. We are also 
considering to generalize PIM to solve other PDEs on manifolds. 

\vspace{0.1in}
\noindent{\bf Acknowledgments.}
This research was partial supported by NSFC Grant (11201257 to Z.S., 11371220 to Z.S. and J.S. and 11271011 to J.S.), 
and  National Basic Research Program of China (973 Program 2012CB825500 to J.S.).

\bibliographystyle{abbrv}
\bibliography{poisson}
\end{document}